\documentclass[12pt]{amsart}

\usepackage{amsfonts}
\usepackage{amsmath,amssymb,amsfonts}
\usepackage{color}
\usepackage{hyperref}
\usepackage[latin1]{inputenc}
\usepackage{amsthm}
\usepackage{graphicx}
\usepackage{tikz}
\usepackage{mathabx}

\def\squarebox#1{\hbox to #1{\hfill\vbox to #1{\vfill}}}
\newcommand{\qedem}{\hspace*{\fill}
\vbox{\hrule\hbox{\vrule\squarebox{.667em}\vrule}\hrule}\smallskip}
\newtheorem{teo}{Theorem}[section]
\newtheorem{prop}[teo]{Proposition}

\newtheorem{coro}[teo]{Corollary}

\newtheorem{lema}[teo]{Lemma}
\newtheorem{defin}[teo]{Definition}

\begin{document}

\title{Cellular Homology of Real Flag Manifolds}

\author{Lonardo Rabelo}
\address{Department of Mathematics, Federal University of Juiz de Fora, Juiz de Fora 36036-900, Minas Gerais, Brazil}
\email{lonardo@ice.ufjf.br}
\thanks{Supported by FAPESP grant number 08/04628-6}

\author{Luiz A. B. San Martin}
\address{Imecc - Unicamp, Departamento de Matem\'atica. Rua S\'ergio Buarque de
Holanda, 651, Cidade Universit\'aria Zeferino Vaz. 13083-859 Campinas S\~ao
Paulo, Brazil}
\email{smartin@ime.unicamp.br}
\thanks{Supported by CNPq grant n$^{\mathrm{o}}$ 305513/2003-6 and FAPESP grant n$^{%
\mathrm{o}}$ 07/06896-5}
\date{}

\begin{abstract}
Let $\mathbb{F}_{\Theta }=G/P_{\Theta }$ be a generalized flag manifold,
where $G$ is a real noncompact semi-simple Lie group and $P_{\Theta }$ a
parabolic subgroup. A classical result says the Schubert cells, which are
the closure of the Bruhat cells, endow $\mathbb{F}_{\Theta }$ with a
cellular CW structure. In this paper we exhibit explicit parametrizations of
the Schubert cells by closed balls (cubes) in $\mathbb{R}^{n}$ and use them
to compute the boundary operator $\partial $ for the cellular homology. We
recover the result obtained by Kocherlakota [1995], in the setting of
Morse Homology, that the coefficients of $\partial $ are $0$ or $\pm 2$ (so
that $\mathbb{Z}_{2}$-homology is freely generated by the cells). In
particular, the formula given here is more refined in the sense that the
ambiguity of signals in the Morse-Witten complex is solved.

\end{abstract}

\maketitle

\noindent

\noindent


\section*{Introduction}

Let $\mathbb{F}_{\Theta }=G/P_{\Theta }$ be a flag manifold of the
non-compact semi-simple Lie group $G$ where $P_{\Theta }$ is a parabolic
subgroup. A classical result says that a cellular structure of $\mathbb{F}%
_{\Theta }$ is given by the Schubert cells $\mathcal{S}_{w}^{\Theta }$ which
are the closure of the Bruhat cells, that is, the components of the Bruhat
decomposition 
\begin{equation*}
\mathbb{F}_{\Theta }=\coprod_{w\in \mathcal{W}/\mathcal{W}_{\Theta }}N\cdot
wb_{\Theta },
\end{equation*}%
where $N$ is the nilpotent component of the Iwasawa decomposition, $\mathcal{%
W}$ is the Weyl group of the corresponding Lie algebra and $\mathcal{W}_{\Theta}$ is the subgroup of $\mathcal{W}$ associated with $\Theta$.

In order to compute the cellular homology of $\mathbb{F}_{\Theta}$, our first task in this paper is to provide explicit parametrizations of the Schubert cells $\mathcal{S}_{w}^{\Theta }$ by cubes $\left[ 0,\pi \right] ^{d}\subset \mathbb{R}^{d}$ which are defined in terms of the reduced decompositions of $w$. This description turns out to be useful to get algebraic formulas for the boundary
operator $\partial $ of the cellular homology. 

Our strategy consists by working firstly in the maximal flag manifolds, denoted by $\mathbb{F}$, and then by projecting down the Schubert cells via the canonical map $\pi _{\Theta }:
\mathbb{F}\rightarrow \mathbb{F}_{\Theta }$. To parametrize a Schubert cell $\mathcal{S}_{w}$, $w\in \mathcal{W}$, in the maximal flag manifold $\mathbb{F}$, we start with a minimal decomposition $w=r_{1}\cdots r_{n}$ of $w$ as a product of reflections $r_{i}=r_{\alpha _{i}}$ with respect to the simple roots. Then, similar to the construction of Bott-Samelson
dessingularization, we see $\mathcal{S}_{w}$ as a product $K_{1}\cdots
K_{n}\cdot b_{0}$, where $b_{0}=P$ is the origin of $\mathbb{F}$ and $K_{i}$
are maximal compact subgroups of rank one Lie groups $G_{i}$ (see Section %
\ref{secschubertmax}). This presents $\mathcal{S}_{w}$ as successive
fibrations by spheres $S^{d_{i}}$, where $d_{i}$ are the multiplicities of
the roots $\alpha _{i}$ - which may be not equal to $1$. Thus a
parametrization $\Phi _{w}:B^{d}\rightarrow 
\mathcal{S}_{w}$ of a cell of dimension $d=d_{1}+\cdots d_{n}$ is obtained by viewing $S^{d_{i}}$
as the ball $B^{d_{i}}$ whose boundary is collapsed to a point.

The case of interest for homology are the roots $\alpha _{i}$ with
multiplicity $d_{i}=1$. This is because the boundary operator $\partial $
for the cellular homology takes the form $\partial \mathcal{S}_{w}=\sum
c\left( w,w^{\prime }\right) \mathcal{S}_{w^{\prime }}$ with $w^{\prime
}=r_{1}\cdots \widehat{r_{i}}\cdots r_{n}$ and the index $i$ is such that $%
d_{i}=1$. In this case, the characteristic map $\Phi _{w}$ is defined in $%
B^{d-1}\times \left[ 0,\pi \right] $ and the coefficient $c\left(
w,w^{\prime }\right) $ is the sum of the degrees of the attaching maps, that
is, the restrictions of $\Phi _{w}$ to $B^{d-1}\times \{0\}$ and $%
B^{d-1}\times \{\pi \}$ (see Section \ref{sechommax}, in particular the example of $\mathrm{Sl}\left( 3,\mathbb{R}\right) $ in Subsection \ref{subsecsl3}). 
This way we get that any coefficient $c\left( w,w^{\prime }\right) $ is $0$
or $\pm 2$. In particular, the $\mathbb{Z}_{2}$-homology is the vector space
with basis $\mathcal{S}_{w}$, $w\in \mathcal{W}$. 

Once the maximal flag manifold is worked out, we get the boundary operator $%
\partial ^{\Theta }$ in a general flag manifold $\mathbb{F}_{\Theta }$. 
Actually we can prove that for a cell $\mathcal{S}_{w}^{\Theta }$ in $\mathbb{F}_{\Theta }$ there exists a unique (minimal) cell $\mathcal{S}_{w}$ in $\mathbb{F}$ with $\pi _{\Theta
}\left( \mathcal{S}_{w}\right) =\mathcal{S}_{w}^{\Theta }$. Then $\partial
^{\Theta }$ is obtained directly from the $\partial $ applied to the minimal
cells.

These results were already obtained by Kocherlakota \cite{Koc95} in the realm
of Morse homology. In \cite{Koc95}, Theorem 1.1.4, it is proved that the
boundary operator for the Morse-Witten complex has coefficients $0$ or $\pm
2 $ as well. Clearly the cellular and the Morse-Witten complexes are
intimately related since the Bruhat cells are the unstable manifolds of the
gradient flow of a Morse function (see Duistermat-Kolk-Varadarajan \cite{DKV83}
). Nevertheless the cellular point of view has the advantage of showing the
geometry in a more evident way. For instance, in the Subsection \ref{secgrad}, 
we provide a description of the flow lines of the gradient flow inside a Bruhat cell
in terms of characteristic maps of the cellular decomposition.
Also, the choice of minimal decompositions for the elements of $\mathcal{W}$ fix 
certain signs that are left ambiguous in the Morse-Witten complex. 

The construction of cellular decompositions of group manifolds and
homogeneous spaces is an old theme. For the classical compact Lie groups one
can build cells using products of reflections via a method that goes back to
Whitehead \cite{Whi44} and was later developed by Yokota \cite{Yok55}, 
\cite{Yok56}. By projection the decomposition on group level induces
decompositions on the Stiefel manifolds $V_{n,k}$, that were exploited by
Miller \cite{Mil53} to get several homological properties of these manifolds.
On the contrary the cellular decompositions of the group manifolds do not
project, in general, to cells in the flag manifolds. Hence that method does not yield cellular
decomposition of the flag manifolds.

On the other hand the Schubert cells are central objects in the study of
(co) homological properties of the 
flag manifolds (see e.g.
Bernstein-Gelfand-Gelfand \cite{BGG73} and references therein). In the complex
case the cellular homology is computed trivially since the cells are all
even dimensional hence boundary operator $\partial =0$ and the homology
groups are freely generated. We refer also to Casian-Stanton \cite{CS99} 
for an approach through representation theory of algebraic reductive groups.

For the real flag manifolds $\partial $ is not, in general, trivial and its
computation requires explicit expressions for the gluing maps between the
cells as we provide in this paper. To the best of our knowledge there is no
systematic construction of the cellular decomposition of the flag manifolds
(of arbitrary semi-simple Lie groups) through the Bruhat cells and their
closures the generalized Schubert cells.

The cells constructed here appeared before (up to cells of dimension two) in
Wiggerman \cite{Wig98}, that uses them to get generators and relations for the
fundamental groups of the flag manifolds. Also in Rabelo \cite{Rab16} and 
Rabelo-Silva \cite{RL18} the method of this paper is used to compute the 
integral homology of the Real isotropic Grassmannians (those of type B,C and D).

The article is organized as follows: In Section \ref{secschubertmax} we construct the parametrizations of the
Schubert cells on the maximal flag manifolds and analyze the attaching
(gluing) maps. In particular, in the subsection \ref{secgrad} we
look at some aspects of the gradient flow yielding Morse homology.
Section \ref{sechommax} is devoted to the boundary operator $\partial $ on the maximal flag manifold. 
The partial flag manifolds are treated in Section \ref{secpartial}.

In this point we would like to thank Lucas Seco for his comments on some
proofs and for his interest in the problem suggesting interesting references
related to this question.

\subsection*{Notation}\label{secpre}

Flag manifolds are defined as homogeneous spaces $G/P$ where $G$ is a
noncompact semi-simple Lie group and $P$ is a parabolic subgroup of $G$. 

Let $\mathfrak{g}$ be a noncompact real semi-simple Lie algebra. The flag
manifolds for the several groups $G$ with Lie algebra $\mathfrak{g}$ are the
same. With this in mind we take always $G$ to be the identity component of
the automorphism group of $\mathfrak{g}$, which is centerless.

Take a Cartan decomposition $\mathfrak{g} = \mathfrak{k} \oplus \mathfrak{s}$
with $\mathfrak{k}$ the compactly embedded subalgebra and denote by $\theta$
the corresponding Cartan involution. Let $\mathfrak{a}$ be a maximal abelian
subalgebra contained in $\mathfrak{s}$ and denote by $\Pi$ the set of roots
of the pair $(\mathfrak{g}, \mathfrak{a})$. Fix a simple system of roots $%
\Sigma \subset \Pi$. Denote by $\Pi^{\pm}$ the set of positive and negative
roots respectively and by $\mathfrak{a}^+$ the Weyl chamber 
\begin{equation*}
\mathfrak{a}^+ = \{ H \in \mathfrak{a} : \alpha (H) > 0 \mbox{ for all }
\alpha \in \Sigma \}.
\end{equation*}

Let $\mathfrak{n} = \displaystyle \sum_{\alpha \in \Pi^+} \mathfrak{g}%
_\alpha $ be the direct sum of root spaces corresponding to the positive
roots. The Iwasawa decomposition of $\mathfrak{g}$ is given by $\mathfrak{g}
= \mathfrak{k} \oplus \mathfrak{a} \oplus \mathfrak{n}$. The notations $K, A$
and $N$ are used to indicate the connected subgroups whose Lie algebras are $%
\mathfrak{k}, \mathfrak{a}$ and $\mathfrak{n}$ respectively.

A sub-algebra $\mathfrak{h} \subset \mathfrak{g}$ is said to be a Cartan sub-algebra if $\mathfrak{h}_{\mathbb{C}}$ is a Cartan sub-algebra of $\mathfrak{g}_{\mathbb{C}}$. If $\mathfrak{h} = \mathfrak{a}$ is a Cartan sub-algebra of $\mathfrak{g}$ we say that $\mathfrak{g}$ is a split real form of $\mathfrak{g}_{\mathbb{C}}$.

A minimal parabolic subalgebra of $\mathfrak{g}$ is given by $\mathfrak{g} = 
\mathfrak{m} \oplus \mathfrak{a} \oplus \mathfrak{n}$ where $\mathfrak{m}$
is the centralizer of $\mathfrak{a}$ in $\mathfrak{k}$. Let $P$ be the
minimal parabolic subgroup with Lie algebra $\mathfrak{p}$ which is the
normalizer of $\mathfrak{p}$ in $G$. We call $\mathbb{F} = G/P$ the maximal
flag manifold of $G$ and denote by $b_0$ the base point $1 \cdot P$ in $G/P$.

Associated to a subset of simple roots $\Theta\subset \Sigma$ there are
several Lie algebras and groups. We write $\mathfrak{g}(\Theta)$ for the
semi-simple Lie algebra generated by $\mathfrak{g}_{\pm \alpha}$, $\alpha
\in \Theta$. Let $G(\Theta)$ be the connected group with Lie algebra $%
\mathfrak{g}(\Theta)$. Moreover, let $\mathfrak{n}_\Theta$ be the subalgebra
generated by the roots spaces $\mathfrak{g}_{-\alpha}$, $\alpha \in \Theta$
and put 
\begin{equation*}
\mathfrak{p}_\Theta = \mathfrak{n}_\Theta \oplus \mathfrak{p}.
\end{equation*}

The normalizer $P_{\Theta }$ of $\mathfrak{p}_{\Theta }$ in $G$ is a
standard parabolic subgroup which contains $P$. The corresponding flag
manifold $\mathbb{F}_{\Theta }$ is called a partial flag manifold of $G$ or
flag manifold of type $\Theta $. We denote by $b_{\Theta }$ the base point $%
1\cdot P_{\Theta }$ in $G/P_{\Theta }$. Such a flag manifold can also be
written as $\mathbb{F}_{\Theta }=K/K_{\Theta }$ where $K_{\Theta }=P_{\Theta
}\cap K$.

The Weyl group $\mathcal{W}$ associated to $\mathfrak{a}$ is the finite
group generated by the reflections over the root hyperplanes $\alpha = 0$
contained in $\mathfrak{a}$, $\alpha \in \Sigma $, and can be alternatively
given as the quotient $M^* / M$ where $M^*$ and $M$ are respectively the
normalizer and the centralizer of $\mathfrak{a}$ in $K$ (the Lie algebra of $%
M$ is $\mathfrak{m}$). We use the same letter to denote a representative of $%
w$ in $M^*$.

For the subset $\Theta \subset \Sigma$, there exists the subgroup $\mathcal{W%
}_\Theta$ which acts trivially on $\mathfrak{a}_\Theta = \{ H \in \mathfrak{a%
} : \alpha (H) = 0 , \alpha \in \Theta\}$. Alternatively, $\mathcal{W}%
_\Theta $ may be seen as the subgroup of the Weyl group generated by the
reflections with respect to the roots $\alpha \in \Theta$.

Viewing the elements of $\mathcal{W}$ as product of simple reflections, the
length $\ell(w)$ of $w \in \mathcal{W}$, is the number of simple reflections in
any reduced expression of $w$ which is equal to the cardinality of $\Pi_w =
\Pi^+ \cap w \Pi^-$, the set of positive roots sent to negative roots by $%
w^{-1}$. If $w = r_1 \cdots r_n$ is a reduced expression of $w$ then 
\begin{equation*}
\Pi_w = \{ \alpha_{1}, r_1 \alpha_2, \ldots, r_1 \cdots r_{n-1}\alpha_n\}.
\end{equation*}

There are two equivalent definitions of order between elements in the Weyl
group (see Humphreys \cite{Hum90}).

\begin{enumerate}
\item First, two elements are connected, denoted $w_{1}\rightarrow w_{2}$,
if $\ell(w_{1})<\ell(w_{2})$ and there is a root $\alpha $ (not necessarily
simple) such that $w_{1}r_{\alpha }=w_{2}$. Now that $w_{1}<w_{2}$ if there
are $u_{1},\ldots ,u_{k}\in \mathcal{W}$ with 
\begin{equation*}
w_{1}\rightarrow u_{1}\rightarrow \cdots u_{k}\rightarrow w_{2}.
\end{equation*}

It may happen that $w_{1}\rightarrow w_{2}$ with $\ell(w_{1})+1=\ell(w_{2})$ but
there is no simple root with $w_{1}r_{\alpha }=w_{2}$.

The definion may be changed by multiplication in the left $r_{\alpha
}w_{1}=w_{2}$ because $r_{\alpha }w_{1}=w_{1}(w_{1}^{-1}r_{\alpha
}w_{1})=w_{1}r_{\beta }$ with $\beta =w^{-1}\alpha $.

\item $w_{1}\leq w_{2}$ if given a reduced expression $w_{2}=r_{1}\cdots
r_{\ell(w_{2})}$ then $w_{1}=r_{i_{1}}\cdots r_{i_{k}}$ for some indices $%
i_{1}<\cdots <i_{k}$.
\end{enumerate}

There is a unique $w_{0}\in \mathcal{W}$ such that $w_{0}\Pi ^{+}=\Pi ^{-}$
which we call the principal involution and is the maximal element in the
Bruhat-Chevalley order.

A partial flag manifold is the base space for the natural equivariant
fibration $\pi _{\Theta }:\mathbb{F}\rightarrow \mathbb{F}_{\Theta }$ whose
fiber is $P_{\Theta }/P$. This fiber is a flag manifold of a semi-simple Lie
group $M_{\Theta }\subset G$ whose rank is the order of $\Theta $. The Weyl
group of $M_{\Theta }$ is the subgroup $\mathcal{W}_{\Theta }$. Its orbit
through $b_{0}$ is contained in the fiber $\pi _{\Theta }^{-1}\pi _{\Theta
}(b_{0})$.

In particular, the group $M_{\Theta }$ is of rank one if $\Theta $ is a
singleton. For example, if $\alpha $ is a simple root, the fiber of $\mathbb{%
F}\rightarrow \mathbb{F}_{\alpha }=G/P_{\alpha }$ which is $P_{\alpha }/P$,
coincides with the (unique) flag manifold of the group $G(\alpha )$ whose
Lie algebra is $\mathfrak{g}(\alpha )$, generated by $\mathfrak{g}_{-\alpha
} $ and $\mathfrak{g}_{\alpha }$. These rank one flag manifolds are spheres $%
S^{m}$, where $m=\dim (\mathfrak{g}_{\alpha }+\mathfrak{g}_{2\alpha })$.

The Bruhat decomposition presents the flag manifolds as a union of $N$%
-orbits (or one of its conjugates). It says that the $N$-orbits on a flag
manifold $\mathbb{F}_{\Theta}$ is finite and coincide with the orbits that
goes through the $A$-fixed points.

\begin{prop}
Let $b_{\Theta }$ be the origin of $\mathbb{F}_{\Theta }$. Then the set $A$%
-fixed points coincides with the orbit $M^{\ast }b_{\Theta }$. This set is
finite and is in bijection with $\mathcal{W}/\mathcal{W}_{\Theta }$.
\end{prop}

Thus the Bruhat decomposition reads 
\begin{equation*}
\mathbb{F}_{\Theta }=\coprod_{w\in \mathcal{W}/\mathcal{W}_{\Theta }}N\cdot
wb_{\Theta }\quad, \quad w\in M^{\ast },
\end{equation*}%
where $N\cdot w_{1}b_{\Theta }=N\cdot w_{2}b_{\Theta }$ if $w_{2}\mathcal{W}%
_{\Theta }=w_{1}\mathcal{W}_{\Theta }$. When there is an equivariant
fibration $\mathbb{F}_{\Theta _{1}}\rightarrow \mathbb{F}_{\Theta _{2}}$ (in
particular when $\mathbb{F}_{\Theta _{1}}=\mathbb{F}$) the $N$-orbits
project onto $N$-orbits by equivariance, hence the fibration respects the
Bruhat decompositions.

Each $N$-orbit through $w$ is diffeormophic to an Euclidean space. Such an
orbit $N\cdot wb_{\Theta }$ is called a Bruhat cell. Its dimension is given
by the formula 
\begin{equation*}
\dim \left( N\cdot wb_{\Theta }\right) =\displaystyle\sum_{\alpha \in \,\Pi
_{w}\,\setminus \,\langle \Theta \rangle }m_{\alpha }
\end{equation*}
where $m_{\alpha }$ is the multiplicity of the root space $\mathfrak{g}%
_{\alpha }$ and $\langle \Theta \rangle $ denotes the roots in $\Pi $
generated by $\Theta $ (see the Lemma \ref{lemproperties} for the maximal
flag case and Lemma \ref{minimal_element} for the partial flag case). In
particular, the Bruhat cell $N\cdot w_{0}b_{\Theta }$ is an open and dense
orbit. The closure of the Bruhat cells are called (generalized) Schubert cells.

\begin{defin}
A Schubert cell is the closure of a Bruhat Cell: 
\begin{equation*}
\mathcal{S}_{w}^{\Theta }=\mathrm{cl}(N\cdot wb_{\Theta }).
\end{equation*}
\end{defin}

The Schubert cells endow the flag manifolds with a cellular decomposition. For a maximal flag manifold we avoid the superscript $\Theta $ and write
simply 
\begin{equation*}
\mathcal{S}_{w}=\mathrm{cl}(N\cdot wb_{0})
\end{equation*}

We recall the following well known facts (see \cite{DKV83} or 
Warner \cite{War72}).

\begin{prop}
\label{BC-order} $\mathcal{S}_{w_{1}}^{\Theta }\subset \mathcal{S}%
_{w_{2}}^{\Theta }$ if and only if $w_{1}\leq w_{2}$.
\end{prop}

\begin{prop}
\label{BC-order-1} $\mathcal{S}_{w}^{\Theta }=\displaystyle{\cup _{u\leq
w}N\cdot ub_{\Theta }}$.
\end{prop}

In the forthcoming sections we will look carefully at the cellular
decompositions of the flag manifolds given by the Schubert cells. Before
going into them we present examples showing that classical cell
decompositions of compact groups are not well behaved with respect to
projections to flag manifolds.

\vspace{12pt}

\noindent
\textbf{Example:} 
In the cellular decomposition of $\mathrm{SO}\left(
3\right) $ of \cite{Whi44} and \cite{Mil53} there are $4$ cells of dimensions 
$0$, $1$, $2$ and $3$. The $2$-dimensional cell is given by the map $%
f:\mathbb{RP}^2 \rightarrow \mathrm{SO}\left( 3\right) $ given by $f\left( \left[ x\right]
\right) =r_{x}d$, $x\in \mathbb{R}^{3}\setminus \{0\}$, where $r_{x}$ is the
reflection in $\mathbb{R}^{3}$ with respect to the plane orthogonal to $x$
and $d=\mathrm{diag}\{1,1,-1\}$ needed to correct the determinant. This map
is viewed as a two-cell $B_{2}\rightarrow \mathrm{SO}\left( 3\right) $ by
taking the interior of the $2$-ball $B_{2}$ as the set $\{\left[ x\right]
\in \mathbb{RP}^{2}:x_{3}\neq 0\}$ where $x=\left( x_{1},x_{2},x_{3}\right) $%
. The boundary of $B_{2}$ is mapped to the $1$-dimensional cell which is the
image under $f$ of $\mathbb{RP}^{1}=\{\left[ \left( x_{1},x_{2},0\right) %
\right] \in \mathbb{RP}^{2}\}$. If $\{e_{1},e_{2},e_{3}\}$ is the standard
basis of $\mathbb{R}^{3}$ then $f\left( \left[ e_{1}\right] \right) =\mathrm{%
diag}\{-1,1,-1\}$, $f\left( \left[ e_{2}\right] \right) =\mathrm{diag}%
\{1,-1,-1\}$ and $f\left( \left[ e_{3}\right] \right) =\mathrm{id}$. These
three elements belong to the group $M$ where $\mathbb{F}=\mathrm{SO}\left(
3\right) /M$ is the maximal flag manifold of $\mathrm{Sl}\left( 3,\mathbb{R}%
\right) $. Hence the projection to $\mathbb{F}$ of the $2$-cell in $\mathrm{%
SO}\left( 3\right) $ is not a cell in $\mathbb{F}$ because $f\left( \left[
e_{i}\right] \right) $, $i=1,2,3$ are projected to the same point, namely
the origin of $\mathbb{F}$.

For other examples we recall the cellular decomposition of $\mathrm{SU}%
\left( n\right) $ given in \cite{Yok56}, Theorem 7.2, where the
positive dimensional cells have dimension $\geq 3$. Hence this construction
does not yield, by projection $\mathrm{SU}\left( n\right) \rightarrow 
\mathrm{SU}\left( n\right) /H$, a cellular decomposition of $\mathrm{SU}%
\left( n\right) /H$ if this homogeneous space has non trivial homology at
the levels $1$ or $2$. This happens, for instance, with the flag manifolds
of $\mathrm{Sl}\left( n,\mathbb{C}\right) $, that have nontrivial $H_{2}$.
Also, the maximal compact subalgebra of the split real form of the
exceptional type $E_{7}$ is $\mathfrak{su}\left( 8\right) $. However the
maximal flag manifold of a split real form has nontrivial fundamental group
(and hence $H_{1}$) as follows by Johnson \cite{Joh04} and \cite{Wig98}
. 

\section{Schubert cells in maximal flag manifolds\label{secschubertmax}}

In this section we give a detailed description of the Schubert cells in the
maximal flag manifolds. This description includes a parametrization by
compact groups (subsets of them) which allows explicit expressions for the
gluing maps between the cells. The partial flag manifolds will be trated in the 
Section \ref{secpartial}.

\subsection{Schubert cells and product of compact subgroups}

The main result here is a suitable parametrization for the Schubert cells
which is the basis for the computation of the boundary operator for the
cellular homology.

As before, $\mathbb{F}=G/P$ is the maximal flag manifold. We denote by $%
\mathbb{F}_{i}=G/P_{i}$ the partial flag manifolds where $P_{i}=P_{\{\alpha
_{i}\}}$, with $\alpha _{i}$ a simple root. The canonical fibration is $\pi
_{i}:\mathbb{F}\rightarrow \mathbb{F}_{i}$.

The Schubert cells are firstly described by the \textquotedblleft
fiber-exhausting\textquotedblright\ map $\gamma _{i}$ defined by 
\begin{equation*}
\gamma _{i}(X)=\pi _{i}^{-1}\pi _{i}(X)\quad,\quad X\subset \mathbb{F},
\end{equation*}%
that is, $\gamma _{i}\left( X\right) $ is the union of the fibers of $\pi
_{i}:\mathbb{F}\rightarrow \mathbb{F}_{i}$ crossing $X\subset \mathbb{F}$.
Notice that each $\gamma _{i}$ is an equivariant map, i.e., $g\gamma _{i}\left( X\right) =\gamma
_{i}\left( gX\right) $, for all $g\in G$ and $X\subset \mathbb{F}$, since the projections $%
\pi _{i}$ are equivariant maps.

For $w\in \mathcal{W}$, put $N^{w}=wNw^{-1}$.

Every Schubert cell is the image of some $g\in G$ of $\mathrm{cl}\left(
N^{w}wb_{0}\right) $. The following result was proved in \cite{San98}.

\begin{teo}
Let $w=r_{1}\cdots r_{n}$ be a reduced expression of $w\in \mathcal{W}$ as a
product of reflections with respect to the simple roots. Then, for any $%
k=1,\ldots ,n$, we have 
\begin{equation*}
\mathrm{cl}(N^{w}b_{0})=\gamma _{1}\cdots \gamma _{k}\left( \mathrm{cl}%
\left( N^{w}r_{1}\cdots r_{k}b_{0}\right) \right) .
\end{equation*}
\end{teo}

In particular, for $k=n$ we have 
\begin{equation}
\mathrm{cl}(N^{w}b_{0})=\gamma _{1}\cdots \gamma _{n}\left( \mathrm{cl}%
\left( wNw^{-1}wb_{0}\right) \right) =\gamma _{1}\cdots \gamma _{n}\{wb_{0}\}
\label{forSchubert0}
\end{equation}%
because $Nb_{0}=b_{0}$. From this equality we get the following expression for the Schubert cell $%
\mathcal{S}_{w}$.

\begin{coro}
\label{corschubexaust}Let $w=r_{1}\ldots r_{n}$ be a reduced expression as a
product of reflections with respect to the simple roots in $\Sigma $. Then, 
\begin{equation*}
\mathcal{S}_{w}=\gamma _{n}\cdots \gamma _{1}\{b_{0}\}
\end{equation*}%
(Note that the order of the indexes is reversed.)
\end{coro}

\begin{proof}
We have $\mathrm{cl}(Nw\cdot b_{0})=w\left( \mathrm{cl}(N^{w^{-1}}b_{0})%
\right) $, hence by (\ref{forSchubert0}) with $w^{-1}=r_{n}\cdots r_{1}$
instead of $w$ we have 
\begin{equation*}
\mathcal{S}_{w}=w\gamma _{n}\cdots \gamma _{1}(w^{-1}b_{0})=\gamma
_{n}\cdots \gamma _{1}\{b_{0}\},
\end{equation*}%
where the last equality follows by equivariance.
\end{proof}

Now we change slightly the above expression in terms of exhausting-fiber
maps to get the Schubert cells as unions of successive orbits of the
parabolic subgroups $P_{i}$. This construction is in the same spirit as the
Bott-Samelson dessingularization (see \cite{DKV83}).

It starts with the remark that the fiber $\gamma _{i}\{b_{0}\}$ of $\pi _{i}:%
\mathbb{F}\rightarrow \mathbb{F}_{i}$ through the origin is the orbit $%
P_{i}\cdot b_{0}$. In general, the fiber through $g\cdot b_{0}\in \mathbb{F}$
is given by $g\cdot \gamma _{i}\{b_{0}\}$ by equivariance of $\gamma _{i} $.
Now, if we have two iterations $\gamma _{2}\gamma _{1}$, then by equivariance we
get 
\begin{eqnarray}
\gamma _{2}\gamma _{1}\{b_{0}\} &=&\gamma _{2}\left( \bigcup_{g\in
P_{1}}g\cdot b_{0}\right)  \label{union} \\
&=&\left( \bigcup_{g\in P_{1}}g\cdot \gamma _{2}(b_{0})\right)  \notag \\
&=&\left( \bigcup_{g\in P_{1}}g\cdot \left( P_{2}b_{0}\right) \right)  \notag
\\
&=&P_{1}P_{2}\cdot b_{0}.  \notag
\end{eqnarray}

Proceeding successively by induction, we obtain 
\begin{equation*}
\mathcal{S}_{w}=\gamma _{n}\cdots \gamma _{1}\{b_{0}\}=P_{1}\cdots
P_{n}\cdot b_{0},
\end{equation*}%
where the indexes of $P_{1}\cdots P_{n}$ is the same as those of minimal
decomposition $w=r_{i}\ldots r_{n}\in \mathcal{W}$.

The same expression still holds with the compact $K_{i}=K\cap P_{i}$ instead
of $P_{i}$. In fact, $K_{i}\cdot b_{0}=P_{i}\cdot b_{0}$ by the Langlands
decomposition $P_{i}=K_{i}AN$ and $AN\cdot b_{0}$. Hence the same arguments
yield the following description of the Schubert cells.

\begin{prop}
\label{Schubert_2}Let $w=r_{1}\cdots r_{n}$ be a reduced expression as a
product of reflections with respect to the simple roots in $\Sigma $. Then, 
\begin{equation*}
\mathcal{S}_{w}=K_{1}\cdots K_{n}\cdot b_{0}.
\end{equation*}%
(Here, different from Corollary \ref{corschubexaust}, the indexes of the $%
r_{i}$'s and $K_{i}$'s are in the same order).
\end{prop}

\noindent
\textbf{Remark:} In general, there is more than one reduced expression for $%
w\in \mathcal{W}$, which provides distinct compact subgroups $K_{i}$ and
distinct parametrizations.

\vspace{12pt}
\noindent
\textbf{Example:} Let $G=\mathrm{Sl}(n,\mathbb{R})$ with $\mathfrak{g}=%
\mathfrak{s}\mathfrak{l}(n,\mathbb{R})$. The simple roots are given by $%
\alpha _{i,i+1}=\lambda _{i}-\lambda _{i+1}$. The compact group $K_{i}$
associated to the simple root $\alpha _{i,i+1}$ is given by the rotations

\begin{equation*}
R_{i}^{t}=\exp (tA_{i,i+1})=\left( 
\begin{array}{cccccc}
1 &  &  &  &  &  \\ 
& \ddots &  &  &  &  \\ 
&  & \cos t & \sin t &  &  \\ 
&  & -\sin t & \cos t &  &  \\ 
&  &  &  & \ddots &  \\ 
&  &  &  &  & 1 \\ 
&  &  &  &  & 
\end{array}%
\right)
\end{equation*}%
where $A_{i,i+1}=E_{i,i+1}-E_{i+1,i}$. In this case, a Schubert cell has the
form 
\begin{equation*}
\mathcal{S}_{w}=R_{i_{1}}^{t_{1}}\cdots R_{i_{m}}^{t_{m}}\cdot b_{0},
\end{equation*}%
that is, is the image of the map $(t_{1},\ldots ,t_{m})\mapsto
R_{i_{1}}^{t_{1}}\cdots R_{i_{m}}^{t_{m}}\cdot b_{0}\in \mathbb{F}$.

Continuing with the example, let $n=3$, with $\mathcal{W}$ the permutation
group in three letters. The Schubert cell $\mathcal{S}_{(13)}$ is the whole
flag $\mathbb{F}_{1,2}^{3}$ since $(13)$ is the principal involution. If we
decompose $(13)=(12)(23)(12)$, $\mathcal{S}_{(13)}$ may be parametrized as: 
\begin{equation*}
\left( 
\begin{array}{ccc}
\cos t_{1} & \sin t_{1} &  \\ 
-\sin t_{1} & \cos t_{1} &  \\ 
&  & 1 \\ 
&  & 
\end{array}%
\right) \left( 
\begin{array}{ccc}
1 &  &  \\ 
& \cos t_{2} & \sin t_{2} \\ 
& -\sin t_{2} & \cos t_{2} \\ 
&  & 
\end{array}%
\right) \left( 
\begin{array}{ccc}
\cos t_{3} & \sin t_{3} &  \\ 
-\sin t_{3} & \cos t_{3} &  \\ 
&  & 1 \\ 
&  & 
\end{array}%
\right) \cdot b_{0}.
\end{equation*}%
If we choose to write $(13)=(23)(12)(23)$ we parametrize $\mathcal{S}_{(13)}$
as: 
\begin{equation*}
\left( 
\begin{array}{ccc}
1 &  &  \\ 
& \cos t_{1} & \sin t_{1} \\ 
& -\sin t_{1} & \cos t_{1} \\ 
&  & 
\end{array}%
\right) \left( 
\begin{array}{ccc}
\cos t_{2} & \sin t_{2} &  \\ 
-\sin t_{2} & \cos t_{2} &  \\ 
&  & 1 \\ 
&  & 
\end{array}%
\right) \left( 
\begin{array}{ccc}
1 &  &  \\ 
& \cos t_{3} & \sin t_{3} \\ 
& -\sin t_{3} & \cos t_{3} \\ 
&  & 
\end{array}%
\right) \cdot b_{0}.
\end{equation*}

In these examples the parameter $t_{i}$ range in the interval $[0,\pi ]$
because $R_{i}^{\pi }\cdot b_{0}=b_{0}$ for any $i$ ($b_{0}=\left(
V_{1}\subset V_{2}\right) $ where $V_{1}$ is the one dimensional subspace of 
$\mathbb{R}^{3}$ spanned by the first basic vector and $V_{2}$ is spanned by
the first two basic vectors). This is a general feature since our cell maps
will be defined in cubes $\left[ 0,\pi \right] ^{m}$.

\subsection{Bruhat cells inside the Schubert cell}

The next results determine the points of a Schubert cell $\mathcal{S}%
_{w}=K_{1}\cdots K_{n}\cdot b_{0}$ which are in the corresponding Bruhat cell $%
N\cdot wb_{0}$. 

\begin{lema}
\label{Schubert_lema} Let $w=r_{1}\cdots r_{n-1}r_{n}$ a minimal
decomposition. Define $v=wr_{n}=r_{1}\cdots r_{n-1}$. Let the parabolic
subgroup $P_{n}=P_{\{\alpha _{n}\}}$ with $r_{n}$ the reflection with
respect to $\alpha _{n}$ and $\mathbb{F}_{n}=G/P_{n}$. Let $\pi _{n}:\mathbb{%
F}\rightarrow \mathbb{F}_{n}$ be the canonical projection and denote by $%
b_{n}$ the origin of $G/P_{n}$. Then we have the disjoint union%
\begin{equation}  \label{main_fiber_equation}
\pi _{n}^{-1}(N\cdot wb_{n})=(N\cdot wb_{0})\,\dot{\cup}\,(N\cdot vb_{0})
\end{equation}
\end{lema}

\begin{proof}
The fiber $\pi _{n}^{-1}(wb_{n})$ is the flag manifold of a rank one group.
Its Bruhat decomposition reads 
\begin{equation*}
\pi _{n}^{-1}(wb_{n})=\{vb_{0}\}\,\dot{\cup}\,\left( \pi
_{n}^{-1}(wb_{n})\cap (N\cdot wb_{0})\right) .
\end{equation*}%
Indeed, in $\pi _{n}^{-1}(wb_{n})$ there are a $0$-cell which is $\{vb_{0}\}$
and an open cell. This latter one is $\pi _{n}^{-1}(wb_{n})\cap (N\cdot
wb_{0})$ because it is contained in the Bruhat cell $N\cdot wb_{0}$ and $%
vb_{0}\notin N\cdot wb_{0}$.

The result follows by acting $N$. In fact, $wb_0 \in \pi_n^{-1}(wb_n)
\cap (N \cdot wb_0)$, hence $N \cdot wb_0 = N \left( \pi_n^{-1}(wb_n) \cap
(N \cdot wb_0) \right)$. Also, by equivariance of $\pi_n$ we get $N
\pi_n^{-1}(wb_n) = \pi_n^{-1}(N\cdot wb_n)$. Then, 
\begin{equation*}
\pi_n^{-1}(N\cdot wb_n) = N \left( \{ vb_0 \} \, \dot{\cup} \, (
\pi_n^{-1}(wb_n) \cap (N \cdot wb_0) ) \right) = ( N \cdot vb_0 ) \, \dot{%
\cup} \, (N \cdot wb_0).
\qedhere
\end{equation*}
\end{proof}

We notice that Equation (\ref{main_fiber_equation}) is equivalent to 
\begin{eqnarray*}
\pi_n^{-1}(N \cdot wb_n) &=& \mathcal{S}_v \, \dot{\cup} \, (N \cdot wb_0).
\end{eqnarray*}
because $\pi_n^{-1}(N \cdot wb_n) \cap \mathcal{S}_v = N\cdot vb_0$ and $%
\pi_n^{-1}(N \cdot wb_n) = \pi_n^{-1}(N \cdot vb_n)$ since $wb_n = vb_n$ in $%
\mathbb{F}_n$.

\begin{prop}
\label{propSchubert_3}Write $\mathcal{S}_{w}=K_{1}\cdots K_{n}\cdot b_{0}$.
Take $b=u_{1}\cdots u_{n}\cdot b_{0}$, with $u_{i}\in K_{i}$. Then $b\in 
\mathcal{S}_{w}\setminus N\cdot wb_{0}$ if and only if $u_{i}\in M$ for some 
$i=1,\ldots ,n$.

In other words, an element $b \in \mathcal{S}_{w}$ is inside the Bruhat cell 
$N\cdot wb_{0}$ if and only if there is no $u_{i}\in M$.
\end{prop}

\begin{proof}
Suppose that $u=u_{i}\in M$ for some $i$. Then $u\in K_{j}$ for all $j$,
since $M\subset K_{j}$, so that $v_{j}=uu_{j}u^{-1}\in K_{j}$. Hence $b$ can
be rewritten as $b=u_{1}\cdots u_{i-1}v_{i+1}\cdots v_{n}u\cdot b_{0}$.
Since $ub_{0}=b_{0}$, it follows that $b\in \mathcal{S}_{v}$, with $%
v=r_{1}\cdots \hat{r_{i}}\cdots r_{n}$, which implies that $b\notin N\cdot
wb_{0}$ since $v<w$ and $\mathcal{S}_{w}\setminus N\cdot wb_{0}=\cup _{u<w}%
\mathcal{S}_{u}$.

For the converse we use induction on the length of $w$. If $w = r_1$ has
length one, then the Schubert cell is $\mathcal{S}_{r_1} = {b_0} \cup
(N\cdot u_1b_0)$. So if $u_1\notin M $, then $u_1b_0 \neq b_0$ and hence $%
u_1\cdot b_0 \in N\cdot u_1b_0$.

For $n>1$, let $b=u_{1}\cdots u_{n}\cdot b_{0}$ with $u_{i}\notin M$. We
must show that $b\in N\cdot wb_{0}$. Put $x=u_{1}\cdots u_{n-1}\cdot b_{0}$.
Note that $b\neq x$ for otherwise $u_{n}b_{0}=b_{0}$ which gives $u_{n}\in M$%
, contradicting the assumption.

The induction hypothesis says that $x \in N \cdot vb_0$, $v=r_1 \cdots
r_{n-1}$. Moreover, $\pi_n(b_0) = \pi_n(u_nb_0)$ which implies that $%
\pi_n(x) = \pi_n(b)$, that is, $x$ and $b$ are in the same fiber of $\pi_n$.
Hence $\pi_n(b) \in \pi_n(N\cdot wb_n)$, so that by Lemma \ref{Schubert_lema}%
, $b \in (N\cdot vb_0) \cup (N\cdot wb_0)$.

Now $b \notin N\cdot vb_0$ for otherwise $b = x$. In fact, as $\pi_n(b) =
\pi_n(x) = zb_n $, for some $z \in N$, we have $b \in \pi_n^{-1}(zb_n) \cap
N \cdot vb_0 $. Since this intersection reduces to ${\ zb_0 }$ we have $x =
zb_0$, because $x \in N \cdot vb_0$. Hence $b \in N\cdot wb_0$, concluding
the proof.
\end{proof}

\subsection{Parametrization of subsets of compact subgroups}

The next step is to find subsets of the subgroups $K_{i}$ that cover $%
\mathcal{S}_{w}=K_{1}\cdots K_{n}\cdot b_{0}$ and thus find parametrizations
of the cells.

The fiber $P_{i}/P$ of the projection $\mathbb{F}\rightarrow \mathbb{F}_{i}$
is the flag of the rank one Lie group $G(\alpha )$ whose Lie algebra is $%
\mathfrak{g}(\alpha )$, generated by $\mathfrak{g}_{\pm \alpha }$. The flag
manifold $\mathbb{F}_{\alpha }$ of $G(\alpha )$ is a sphere $S^{m}$ with
dimension $m=\dim \mathfrak{s}_{\alpha }-1$ where $\mathfrak{s}_{\alpha }$
is the symmetric part of the Cartan decomposition of $\mathfrak{g}\left(
\alpha \right) $. If $\{1,w\}$ is the Weyl group of $G\left( \alpha \right) $
and $b_{0}$ is the origin of $\mathbb{F}$ then $b_{0}$ and $wb_{0}$ are
antipodal points in $S^{m}$. The parametrization we seek is provided by the
following lemma whose general proof is only sketched below. In the sequel we
write down the details for the case when $\dim \mathfrak{g}_{\alpha }=1$ and 
$\mathfrak{g}_{2\alpha }=\{0\}$ so that $\mathfrak{g}\left( \alpha \right)
\approx \mathfrak{sl}\left( 2,\mathbb{R}\right) $.

\begin{lema}
\label{parametrization}Let $G=G(\alpha )$ be a real one rank group of rank
with maximal compact subgroup $K=K_{\alpha }$ and the corresponding flag
manifold $\mathbb{F}=S^{m}$ with origin $b_{0}$. Let $B^{m}$ be the closed
ball in $\mathbb{R}^{m}$. Then, there exists a continuous map $\psi
:B^{m}\rightarrow K$ such that

\begin{itemize}
\item $\psi (S^{m-1}) \subset M$ and hence $\psi(S^{m-1}) \cdot b_0 = b_0$.

\item If $x\in B^{m}\setminus S^{m-1}$, then $\psi (x)\cdot wb_{0}$ is a
diffeomorphism onto the Bruhat cell which is the complement of $b_0$.
\end{itemize}
\end{lema}

For the proof of \ the lemma recall that the following list exhaust the rank
one Lie algebras (see \cite{War72}, pages 30-32). In the list $d_{\alpha
}=\dim \mathfrak{g}_{\alpha }$ and $d_{2\alpha }=\dim \mathfrak{g}_{2\alpha
} $.

\begin{itemize}
\item $\mathfrak{so}\left( 1,n\right) $; $d_{\alpha }=n-1$, $d_{2\alpha }=0$%
; $\dim \mathfrak{s}=n$. (This class includes $\mathfrak{sl}\left( 2,\mathbb{%
R}\right) \approx \mathfrak{sp}\left( 1,\mathbb{R}\right) \approx \mathfrak{%
so}\left( 1,2\right) $, $\mathfrak{sl}\left( 2,\mathbb{C}\right) \approx 
\mathfrak{so}\left( 1,3\right) $ and $\mathfrak{su}^{\ast }\left( 4\right)
\approx \mathfrak{so}\left( 1,5\right) $.)

\item $\mathfrak{su}\left( 1,n\right) $; $d_{\alpha }=2\left( n-1\right) $, $%
d_{2\alpha }=1$; $\dim \mathfrak{s}=2n$. (This class includes $\mathfrak{so}%
^{\ast }\left( 6\right) \approx \mathfrak{su}\left( 1,3\right) $.)

\item $\mathfrak{sp}\left( 1,n\right) $; $d_{\alpha }=4\left( n-1\right) $, $%
d_{2\alpha }=3$; $\dim \mathfrak{s}=4n$.

\item A real form of the exceptional Lie algebra $F_{4}$; $d_{\alpha }=8$, $%
d_{2\alpha }=7$; $\dim \mathfrak{s}=16$.
\end{itemize}

The exceptional algebra $F_{4}$ does not appear as a $\mathfrak{g}\left(
\alpha \right) $ in any Lie algebra different from itself because apart
from $F_{4}$ the multiplicities $d_{2\alpha }$ are at most $3$ (see \cite{War72}, pages 30-32). Hence, we can discard it.

On the other hand the classical groups $\mathrm{SO}(1,n)$, $\mathrm{SU}(1,n)$
and $\mathrm{Sp}(1,n)$ contain the compact subgroups $\mathrm{SO}(n)$, $%
\mathrm{SU}(n)$ and $\mathrm{Sp}(n)$ whose actions on the respective flag
manifolds $S^{n-1}$, $S^{2n-1}$ and $S^{4n-2}$ are the standard ones coming
from the linear actions in $\mathbb{R}^{n}$, $\mathbb{C}^{n}$ and $\mathbb{H}%
^{n}$, respectively. In each case the origin $b_{0}$ of the flag is the
first basic vector $e_{1}$ while $wb_{0}=-e_{1}$. Now, take matrices 
\begin{equation*}
A_{\gamma }=\left( 
\begin{array}{cc}
0 & -\overline{\gamma }^{T} \\ 
\gamma & 0%
\end{array}%
\right)
\end{equation*}%
with $\gamma $ in $\mathbb{R}^{n-1}$, $\mathbb{C}^{n-1}$ and $\mathbb{H}%
^{n-1}$ respectively, such that $\left\Vert \gamma \right\Vert =1$. If $%
m=n-1 $, $2n-1$ or $4n-1$ then $U$ is one of the groups $\mathrm{SO}(n)$, $%
\mathrm{SU}(n)$ and $\mathrm{Sp}(n)$. If $U = \mathrm{SO}(n)$ then the map $%
\psi :S^{m-1}\times \left[ 0,2\pi \right] \rightarrow U$ given by $\psi
\left( \gamma ,t\right) =e^{tA_{\gamma }}$ satisfies the requirements of
Lemma \ref{parametrization}, because $e^{tA_{\gamma }}\cdot e_{1}=\cos
te_{1}+\sin t\widetilde{\gamma }$ where $\widetilde{\gamma }=\left( 0,\gamma
\right) $. The complex and quaternionic cases are made similarly with slight
modifications. If $U= \mathrm{SU}(n)$, let $B^{2n-2}=\{t\gamma :\left\Vert
\gamma \right\Vert =1$, $t\in \left[ 0,\pi \right] \}$ and define the map $%
\psi :B^{2n-2}\times \left[ -\pi ,\pi \right] \rightarrow \mathrm{SU}\left(
n\right) \subset K $ by $\psi \left( t\gamma ,\theta \right) =e^{tA_{\gamma
}}e^{D_{\theta }}$ where 
\begin{equation*}
D_{\theta }=\mathrm{diag} \left\{ i\theta ,-\frac{i}{n-1}\theta ,\ldots ,-%
\frac{i}{n-1}\theta \right\}.
\end{equation*}%
It follows that $\psi$ is the desired parametrization.
If $U= \mathrm{Sp}(n)$, the map $\psi :B^{4n-4}\times B^{3}\rightarrow 
\mathrm{Sp}\left( n\right) \subset K $ that realizes the parametrization is
defined by $\psi \left( t\gamma ,q\right) =e^{tA_{\gamma }}e^{D_{q}}$, where 
$B^{4n-4}=\{t\gamma \in \mathbb{H}^{n-1}:\left\Vert \gamma \right\Vert =1$, $%
t\in \left[ 0,\pi \right] \}$, $B^{3}=\{q\in i\mathbb{H}:\left\Vert
q\right\Vert \leq \pi \}$ and 
\begin{equation*}
D_{q}=\mathrm{diag}\left\{q,-\frac{1}{n-1}q,\ldots ,-\frac{1}{n-1}q\right\}.
\end{equation*}

From now on we consider the case when $\dim \mathfrak{g}_{\alpha }=1$ and $%
\mathfrak{g}_{2\alpha }=\{0\}$, so that $\mathfrak{g}(\alpha )\approx 
\mathfrak{sl}(2,\mathbb{R})$ and compact Lie algebra of $K_{\alpha }$ is $%
\mathfrak{so}\left( 2\right) $. This is the only relevant case to the
computation of homology of the flag manifolds (c.f. Proposition \ref%
{propappendix}).

Let $\theta $ be the Cartan involution. Take $0\neq X_{\alpha }\in \mathfrak{%
g}_{\alpha }$ and $Y_{\alpha }=\theta (X_{\alpha })\in \mathfrak{g}_{-\alpha
}$ such that $\langle X_{\alpha },Y_{\alpha }\rangle =\frac{2}{\langle
\alpha ,\alpha \rangle }$. Hence, $[X_{\alpha },Y_{\alpha }]=H_{\alpha
}^{\vee }=\frac{2H_{\alpha }}{\langle \alpha ,\alpha \rangle }$. Denote by $%
A_{\alpha }=X_{\alpha }+Y_{\alpha }\in \mathfrak{k}$. The Lie algebra $%
\mathfrak{g}(\alpha )=\mathfrak{g}_{-\alpha }\oplus \langle H_{\alpha
}^{\vee }\rangle \oplus \mathfrak{g}_{\alpha }$ is isomorphic to $\mathfrak{s%
}\mathfrak{l}(2,\mathbb{R})$. Explicitly, write $\rho :\mathfrak{s}\mathfrak{%
l}(2,\mathbb{R})\rightarrow \mathfrak{g}(\alpha )$, with $\rho (H)=H_{\alpha
}^{\vee }$, $\rho (X)=X_{\alpha }$ and $\rho (Y)=Y_{\alpha }$ where%
\begin{eqnarray*}
H=\left( 
\begin{array}{cc}
1 & 0 \\ 
0 & -1%
\end{array}%
\right) ,\,\hspace{1cm}X=\left( 
\begin{array}{cc}
0 & -1 \\ 
0 & 0%
\end{array}%
\right) ,\,\hspace{1cm}Y=\left( 
\begin{array}{cc}
0 & 0 \\ 
1 & 0%
\end{array}%
\right) .
\end{eqnarray*}

This homomorphism extends to a homomorphism $\phi :\mathfrak{s}\mathfrak{l}%
(2,\mathbb{C})\rightarrow \mathfrak{g}_{\mathbb{C}}(\alpha )$. Note that $%
\mathrm{ad}\circ \phi $ is a representation of $\mathfrak{s}\mathfrak{l}(2,%
\mathbb{C})$ in $\mathfrak{g}_{\mathbb{C}}$. As $\mathrm{Sl}(2,\mathbb{C})$
is simply connected, this representation extends to a representation $\Phi $
of $\mathrm{Sl}(2,\mathbb{C})$ in $\mathfrak{g}_{\mathbb{C}}$ and they are
related by $e^{\mathrm{ad}\circ \phi (X)}=\Phi (\exp (X))$ for any $X\in 
\mathfrak{s}\mathfrak{l}(2,\mathbb{C})$. We have 
\begin{eqnarray*}
e^{\mathrm{ad}(\pi A_{\alpha })}=e^{\mathrm{ad}\circ \phi (A)}=\Phi (\exp
(\pi A)),
\end{eqnarray*}%
where $A=X+Y$. But in $\mathrm{Sl}(2,\mathbb{C})$ we have 
\begin{eqnarray*}
\exp (\pi A_{\alpha }^{\prime }) = \exp \left( 
\begin{array}{cc}
0 & -\pi \\ 
\pi & 0%
\end{array}%
\right) &=& \left( 
\begin{array}{cc}
1 & 0 \\ 
0 & -1%
\end{array}%
\right) \\
&=& \exp \left( 
\begin{array}{cc}
i\pi & 0 \\ 
0 & -i\pi%
\end{array}%
\right) = \exp (i\pi H).
\end{eqnarray*}%
Therefore, 
\begin{eqnarray}
e^{\mathrm{ad}(\pi A_{\alpha })} &=&\Phi (\exp ({i\pi H})) = e^{\mathrm{ad}%
\circ \phi (i\pi H)}  \notag \\
&=&e^{\mathrm{ad}(i\pi H_{\alpha }^{\vee })}.
\end{eqnarray}

Put 
\begin{eqnarray*}
m_{\alpha }=\exp (\pi iH_{\alpha }^{\vee })=\exp (\pi A_{\alpha }).
\end{eqnarray*}
Then $m_{\alpha }$ centralizes $A$ ($m_{\alpha }=\exp (\pi iH_{\alpha
}^{\vee })$) and belongs to $K$ ($m_{\alpha }=\exp (\pi A_{\alpha })$).
Hence $m_{\alpha }\in M=Z_{K}(\mathfrak{a})$.

Now consider the curve $\gamma (t)=\exp (tA_{\alpha })\cdot b_{0}$ in the
fiber of $\mathbb{F}\rightarrow G/P_{\alpha }$ through the origin. Since $%
m_{\alpha }\in M$, $\gamma (\pi )=m_{\alpha }b_{0}=b_{0}$. Actually $\gamma
(t)$ covers the fiber in the interval $[0,\pi ]$.

In $\mathfrak{s}\mathfrak{l}(2,\mathbb{R})$ we have that%
\begin{eqnarray*}
\mathrm{Ad}(e^{tA})H &=&\left( 
\begin{array}{cc}
\cos t & -\sin t \\ 
\sin t & \cos t%
\end{array}%
\right) \left( 
\begin{array}{cc}
1 & 0 \\ 
0 & -1%
\end{array}%
\right) \left( 
\begin{array}{cc}
\cos t & \sin t \\ 
-\sin t & \cos t%
\end{array}%
\right) \\
&=&\left( 
\begin{array}{cc}
\cos 2t & \sin 2t \\ 
-\sin 2t & \cos 2t%
\end{array}%
\right) .
\end{eqnarray*}%
That is, $\mathrm{Ad}(e^{tA})H=-\sin 2tX+\cos 2tH+\sin 2tY$. Applying the
formula 
\begin{equation*}
\rho \left( \mathrm{Ad}(e^{tA})H\right) =\mathrm{Ad}(e^{tA})H_{\alpha
}^{\vee }
\end{equation*}
we get $\mathrm{Ad}(e^{tA})H_{\alpha }^{\vee }=-\sin 2tX_{\alpha }+\cos
2tH_{\alpha }^{\vee }+\sin 2tY_{\alpha }.$ This shows that $e^{tA}$
centralizes $H_{\alpha }^{\vee }$ if and only if $t=n\pi $. In particular, $%
e^{tA}\in M$ if and only if $t=n\pi $. Hence, the period of $\gamma $ is
exactly $\pi $. Summarizing,

\begin{lema}
\label{compact_parametrization_1} The one-dimensional version of the Lemma \ref%
{parametrization} is realized by 
\begin{eqnarray*}
\psi :[0,\pi ]\rightarrow K_{\alpha }\,,\,t\mapsto \exp (tA_{\alpha }).
\end{eqnarray*}%
In particular, $\psi (0)=1$ and $\psi (\pi )=m_{\alpha }=\exp (\pi A_{\alpha
})$.
\end{lema}

\vspace{12pt}

\noindent

Moreover, if $X\in \mathfrak{g}_{\beta }$, then: 
\begin{eqnarray*}
\mathrm{Ad}(m_{\alpha })(X)=\mathrm{Ad}(\exp (\pi iH_{\alpha }^{\vee
}))(X)=e^{\mathrm{ad}(\pi iH_{\alpha }^{\vee })}(X)=e^{\pi i\epsilon (\alpha
,\beta )}(X),
\end{eqnarray*}%
where $\epsilon (\alpha ,\beta )=\frac{2\langle \alpha ,\beta \rangle }{%
\langle \alpha ,\alpha \rangle }$ is the Killing number. This implies that

\begin{lema}
\label{compact_parametrization_2} The root spaces $\mathfrak{g}_{\beta }$ are
invariant by the action of $\mathrm{Ad}(m_{\alpha })$ and 
\begin{eqnarray*}
\mathrm{Ad}(m_{\alpha })_{\left\vert \mathfrak{g}_{\beta }\right.
}=(-1)^{\epsilon (\alpha ,\beta )}\mathrm{id}.
\end{eqnarray*}
\end{lema}

\subsection{Gluing cells}

A Schubert cell $\mathcal{S}_{w}$ is obtained from smaller cells $\mathcal{S}%
_{v}$, $v<w$, by gluing a cell of dimension $\dim (N\cdot wb_{0})$. Once
this proccess is done for each $w\in \mathcal{W}$, we get a cellular
decomposition for $\mathbb{F}$ which is explictly given by characterisc maps
and attaching maps. (We follow the terminology of Hatcher \cite{Hat02}: the \textit{%
characteristic map} is defined in a closed ball while the \textit{attaching
map} is the restriction characteristic map to the boundary of the ball.)

In order to define a characteristic map for $\mathcal{S}_{w}$, $w\in 
\mathcal{W}$, we must choose a reduced expression 
\begin{equation*}
w=r_{1}\cdots r_{n}
\end{equation*}%
as product of simple reflections $r_{i}=r_{\alpha _{i}}$. We know that $%
\mathcal{S}_{w}=K_{1}\cdots K_{n}\cdot b_{0}$. By Lemma \ref{parametrization}%
, for each $i$, there exists $\psi _{i}:B^{d_{i}}\rightarrow K_{i}$, where $%
d_{i}$ is the dimension of the fiber of $\mathbb{F}\rightarrow \mathbb{F}%
_{i} $, that is, the dimension of the flag of $G(\alpha _{i})$.

Let $B_{w}=B^{d_{1}}\times \cdots \times B^{d_{n}}$ be the ball with
dimension $d=d_{1}+\cdots +d_{n}$. Then the characteristic map $\Phi
_{w}:B_{w}\rightarrow \mathbb{F}$ is defined by 
\begin{equation*}
\Phi _{w}(t_{1},\ldots ,t_{n})=\psi _{1}(t_{1})\cdots \psi _{n}(t_{n})\cdot
b_{0}.
\end{equation*}

\vspace{12pt}
\noindent
\textbf{Remark:} Distinct decompositions of $w$ yields different
characteristic maps, so the notation $\Phi _{w}$ should include the minimal
decomposition of $w$ (for example, $\Phi _{r_{1}\cdots r_{n}}$). To keep
this simpler notation we shall fix later a choice of minimal expressions for
each $w\in \mathcal{W}$.

\begin{prop}
\label{characteristic_map} Let $w=r_{1}\cdots r_{n}$ be a minimal
decomposition. Let $\Phi _{w}:B_{w}\rightarrow \mathbb{F}$ be the map
defined above and take $\mathbf{t}=(t_{1},\ldots ,t_{n})\in B_{w}$. Then, $%
\Phi _{w}$ is a characteristic map for $\mathcal{S}_{w}$, that is,

\begin{enumerate}
\item $\Phi_w (B_w) \subset \mathcal{S}_w$.

\item $\Phi _{w}(\mathbf{t})\in \mathcal{S}_{w}\setminus N\cdot wb_{0}$ if
and only if $\mathbf{t}\in \partial B_{w}=S^{d-1}$.

\item $\Phi |_{B_{w}^{\circ }}:B_{w}^{\circ }\rightarrow N\cdot wb_{0}$ is a
diffeormorphism ($B_{w}^{\circ }$ is the interior of $B_{w}$).
\end{enumerate}
\end{prop}

\begin{proof}
The first condition holds by construction since $\psi_i(t_i) \in K_i$ and
hence $\Phi _{w}(t_{1},\ldots ,t_{n})=\psi _{1}(t_{1})\cdots \psi
_{n}(t_{n})\cdot b_{0}\in K_{1}\cdots K_{n}\cdot b_{0}=\mathcal{S}_{w}$.

The second statement follows as a consequence of the Proposition \ref{propSchubert_3} by which we have 
that $\psi _{1}(t_{1})\cdots \psi _{n}(t_{n})$ is not in $%
N\cdot wb_{0}$ if and only if some $\psi _{i}(t_{i})\in M$. By Lemma \ref%
{compact_parametrization_1}, this implies that $t_{i}\in \{0,\pi \}=\partial
S^{d_{i}-1}$, that is, $\mathbf{t}\in \partial S^{d-1}$.

Finally, we already have that $\Phi |_{B_{w}^{\circ }}$ is a surjective map.
Let us prove its injectivity by induction on the lenght $l\left( w\right) $
of $w$. If $l\left( w\right) =1$, this is Lemma \ref%
{compact_parametrization_1}. For $l\left( w\right) >1$, suppose that $\Phi
_{w}\left( \mathbf{t}\right) =\Phi _{w}\left( \mathbf{s}\right) $ with $%
\mathbf{t}=\left( t_{1},\ldots ,t_{n}\right) $ and $\mathbf{s}=\left(
s_{1},\ldots ,s_{n}\right) $ in $B_{w}^{\circ }$. Then we claim that $x=y$
where 
\begin{eqnarray*}
x &=&\psi _{1}(t_{1})\cdots \psi _{n-1}(t_{n-1})\cdot b_{0} \\
y &=&\psi _{1}(s_{1})\cdots \psi _{n-1}(s_{n-1})\cdot b_{0}.
\end{eqnarray*}%
In fact, the elements $\psi _{i}(t_{i})$ and $\psi _{i}(s_{i})$ are not in $%
M $, hence by Proposition \ref{propSchubert_3} both $x,y\in N\cdot vb_{0}$, $%
v=r_{1}\cdots r_{n-1}$. Also, $\pi _{n}\left( x\right) =\pi _{n}\left( \Phi
_{w}\left( \mathbf{t}\right) \right) =\pi _{n}\left( \Phi _{w}\left( \mathbf{%
s}\right) \right)= \pi _{n}\left( x\right) $, that is, $x$ and $y$ belong to
the same fiber of $\pi _{n}:\mathbb{F}\rightarrow \mathbb{F}_{n}$. It
follows by Lemma \ref{Schubert_lema} that $x=y$ since $N\cdot vb_{0}$ meets
each fiber of $\pi _{n}$ in a unique point. By the induction hypothesis $%
(t_{1},\ldots ,t_{n-1})=(s_{1},\ldots ,s_{n-1})$, so that $\psi
_{1}(t_{1})\cdots \psi _{n-1}(t_{n-1})=\psi _{1}(s_{1})\cdots \psi
_{n-1}(s_{n-1})$. Applying this to the equality $\Phi _{w}\left( \mathbf{t}%
\right) =\Phi _{w}\left( \mathbf{s}\right) $ we conclude that $\psi
_{n}(t_{n})\cdot b_{0}=\psi _{n}(s_{n})\cdot b_{0}$, which in turn implies
that $t_{n}=s_{n}$, $l\left( r_{n}\right) =1$. Therefore, $\Phi _{w}$ is a
closed continuous and bijective map, hence it is a homeomorphism
(differentiability comes from the construction of the maps $\psi _{i}$).
\end{proof}

As a consequence of the last item of the above proposition, we have the
following construction. Let $d=\dim \mathcal{S}_{w}=\dim N\cdot wb_{0}$. The
sphere $S^{d}$ is the quotient $B_{w}/\partial (B_{w})$ where the boundary
is collapsed to a point. We can do the same with the Schubert cell $\mathcal{%
S}_{w}$. Define $\sigma _{w}=S_{w}/(S_{w}\setminus N\cdot wb_{0})$, i.e.,
the space obtained by identifying the complement of the Bruhat cell $%
\mathcal{S}_{w}\setminus N\cdot wb_{0}$ in $\mathcal{S}_{w}$ to a point. As $%
\Phi _{w}(\partial (B_{w}))\subset S_{w}\setminus N\cdot wb_{0}$, it follows
that $\Phi _{w}$ induces a map $S^{d}\rightarrow \sigma _{w}$ which is a
homeomorphism. The inverse of this homeomorphism will be denoted by 
\begin{equation}
\Phi _{w}^{-1}:\sigma _{w}\rightarrow S^{d}  \label{Sphere}
\end{equation}%
(although this is not the same as the inverse of $\Phi _{w}$).

A very useful data is the determination of pairs $w,w^{\prime
}\in \mathcal{W}$ for which $w^{\prime} \leq w$ and $\dim \mathcal{S}_{w}-\dim \mathcal{S}%
_{w^{\prime }}=1$. 

\begin{prop}
\label{propappendix} Let $w,w^{\prime }\in \mathcal{W}$. The following
statements are equivalent.

\begin{enumerate}
\item $\mathcal{S}_{w^{\prime }}\subset \mathcal{S}_{w}$ and $\dim 
\mathcal{S}_{w}-\dim \mathcal{S}_{w^{\prime }}=1$.

\item If $w=r_{1}\cdots r_{n}$ is a reduced expression of $w\in \mathcal{%
W}$ as a product of simple reflections, then

\subitem (i) $w^{\prime }=r_{1}\cdots \hat{r_{i}}\cdots r_{n}$ is a reduced
expression.

\subitem (ii) If $r_{i}=r_{\alpha _{i}}$ then $\mathfrak{g}(\alpha
_{i})\cong \mathfrak{s}\mathfrak{l}(2,\mathbb{R})$. This is the same as
saying the fiber of $\mathbb{F}\rightarrow \mathbb{F}_{i}$ has dimension $1$.
\end{enumerate}
\end{prop}

\begin{proof}
In fact, $\mathcal{S}_{w^{\prime }}\subset \mathcal{S}_{w}$ if and only if $%
w^{\prime }<w$ in the Bruhat-Chavalley order. In this case if $w=r_{1}\cdots
r_{n}$ and $w^{\prime }=r_{i_{1}}\cdots r_{i_{j}}$ are reduced expressions
then $\dim \mathcal{S}_{w}$ is $\dim \mathcal{S}_{w^{\prime }}$ plus the sum
of the multiplicities of the roots missing in the reduced expression for $%
w^{\prime }$. Hence $\dim \mathcal{S}_{w}-\dim \mathcal{S}_{w^{\prime }}=1$
if and only if $w^{\prime }=r_{1}\cdots \hat{r_{i}}\cdots r_{n}$ and $\alpha
_{i}$ has multiplicity $1$, that is, $\mathfrak{g}(\alpha _{i})\cong 
\mathfrak{s}\mathfrak{l}(2,\mathbb{R})$.
\end{proof}

\vspace{12pt}
\noindent
\textbf{Remark: }Given $w^{\prime }$ as above, the decomposition $w^{\prime
}=r_{1}\cdots \hat{r_{i}}\cdots r_{n}$ is unique. In fact, if $w=r_{1}\cdots
r_{i}\cdots r_{j}\cdots r_{n}$ and $w^{\prime }=r_{1}\cdots r_{i}\cdots \hat{%
r_{j}}\cdots r_{n}$ then $r_{i+1}\cdots r_{j}=r_{i}\cdots r_{j-1}$ which
cannot happen (see \cite{San10}, Chapter 9).

\subsection{Gradient flow\label{secgrad}}

It is known that the vector field $\widetilde{H}$ with flow $\exp tH$
induced on a flag manifold $\mathbb{F}_{\Theta }$ by a regular element $H\in 
\mathfrak{a}^{+}$ is the gradient of a Morse function. For this flow, the
singularities are $wb_{0}$, $w\in \mathcal{W}$, whose unstable and stable
manifolds are Bruhat cells $W^{u}(wb_{0})=N\cdot wb_{0}$ and $%
W^{s}(ub_{0})=N^{-}\cdot ub_{0}$ (see \cite{DKV83}, for details). 

Below we describe these orbits in terms of characteristic
maps of the cellular decomposition constructed above.

Take $w=r_{1}\cdots r_{n}\in \mathcal{W}$ with the characteristic map $\Phi
_{w}$ and let $w^{\prime }=r_{1}\cdots \hat{r_{i}}\cdots r_{n}$ such that $\mathcal{S}_{w^{\prime}}\subset \mathcal{S}_{w}$ and $\dim \mathcal{S}_{w^{\prime }}=\dim \mathcal{S}_{w}-1$, by Proposition \ref{propappendix}.

Note that by construction $wb_{0}=\Phi _{w}(\pi/2,\ldots ,\pi/2,\ldots ,\pi/2)$, that is, $wb_0$ is the image of the center of the cube. Now,
consider the path 
\begin{equation*}
\phi _{i}(t)=\Phi _{w}\left(\pi/2,\ldots ,t,\ldots ,\pi/2 \right)\qquad \, , \, t\in \lbrack 0,\pi ],
\end{equation*}%
where $t$ is in the $i$-th position. Then $\phi (\pi/2)=w\cdot
b_{0} $, $\phi (0)=w^{\prime }\cdot b_{0}$ comes from the $0$-face and $\phi
(\pi )=w^{\prime }\cdot b_{0}$ comes from the $\pi $-face. Below we prove
that the two pieces of $\phi _{i}\left( t\right) $, from $\pi /2$ to $\pi $
and from $0$ to $\pi /2$ (in reversed direction) are the two gradient lines
joining the singularities $w\cdot b_{0}$ and $w^{\prime }\cdot b_{0}$.

In what follows, we write $w^{\prime }=r_{1}\cdots \hat{r_{i}}\cdots
r_{n}=u\cdot v$, i.e., $u=r_{1}\cdots r_{i-1}$ and $v=r_{i+1}\cdots r_{n}$. Put $%
X_{\beta }=\mathrm{Ad}(u)X_{\alpha _{i}}$, $Y_{\beta }=\theta (X_{\beta })$
and $A_{\beta }=X_{\beta }+Y_{\beta }$.

\begin{lema}
\label{negativo} With the above notation, we have that
\begin{equation*}
\phi _{i}(t)=\exp (sA_{\beta })wb_{0}
\end{equation*}
where $s=t+\pi/2\in \lbrack -\pi/2,\pi/2]$ if $t\in \lbrack 0,\pi]$.
\end{lema}

\begin{proof}
The result is a consequence of the following computation.
\begin{eqnarray*}
\ \phi _{i}(t) &=&u\exp (tA_{i})vb_{0} = u\exp (tA_{i})r_{i}r_{i}vb_{0} \\
&=&u\exp (tA_{i})\cdot \exp ((\pi/2)A_{i})r_{i}vb_{0} \\
&=&\exp \left( \left(t+\pi/2 \right)\mathrm{Ad}(u)A_{i}\right) wb_{0} \\
&=&\exp (sA_{\beta })wb_{0}
\end{eqnarray*}
where $\beta =u\alpha _{i}$ implies that $A_{\beta }=\mathrm{Ad}(u)A_{i}$.
\end{proof}

Let us consider $\phi _{i}(s)=\exp (sA_{\beta })wb_{0}$. It follows that $%
\phi _{i}(0)=wb_{0}$, $\phi _{i}(\pm \pi/2)=w^{\prime }b_{0}$.

\begin{lema}
$\exp(t Y_{\beta})wb_0 = wb_0$.
\end{lema}

\begin{proof}
The main idea is translate to the origin. That is 
\begin{equation*}
\exp (tX_{-\beta })wb_{0}=w(w^{-1}\exp (tX_{-\beta }))wb_{0}=w\exp
(tAd(w^{-1}X_{-\beta }))b_{0}.
\end{equation*}%
The root $\beta $ is positive while $w^{-1}\beta $ is negative. As $u\left( 
\mathfrak{g}_{\alpha }\right) =\mathfrak{g}_{u\alpha }$ we have that $
\mathrm{Ad}(w^{-1}X_{-\beta })\in \mathfrak{g}_{-w^{-1}\beta }$. But since $
w^{-1}\beta $ is a negative root, it follows that $-w^{-1}\beta $ is positive. Hence, it
belongs to $\mathfrak{n}^{+}\subset \mathfrak{p}$ and therefore $\exp
(tAd(w^{-1}X_{-\beta }))b_{0}=b_{0}$ for all $t\in \mathbb{R}$.
\end{proof}

\begin{lema}[\cite{Koc95},Lemma 2.4.1]
Let $t=\tan (s),r=-\sin (s)\cos (s),\lambda ={\cos (s)}^{-1}$. Hence 
\begin{equation*}
\phi _{i}(s)=e^{tX_{\beta }}e^{rY_{\beta }}e^{\log (\lambda )H_{\beta
}^{\vee }}wb_{0}.
\end{equation*}
\end{lema}

Note however that both matrices $e^{\log (\lambda )H_{\beta }^{\vee }}$ and $%
e^{rY{\beta }}$ fix $wb_{0}$. This gives that 
\begin{equation}
\phi _{i}(s)=e^{\tan (s)X_{\beta }}wb_{0}\,,\,s\in \left( -\pi/2,\pi/2\right).  \label{fluxo}
\end{equation}%
Finally, we have that 
\begin{equation*}
\lim_{t\rightarrow \pm \infty }\exp ^{tX_{\beta }}wb_{0}=\phi _{i}(\pm \pi/2)=w^{\prime }b_{0}.
\end{equation*}

From (\ref{fluxo}) we get the behaviour of the gradient flow $h^{t}=\exp tH$
with $H\in \mathfrak{a}^{+}$.

Let $s \neq 0$. It is easy to see that $h^t (\phi_i(s)) = \exp (\tan(s) e^{t
\beta(H)}X_\beta) \cdot wb_0$. This may be written as $h^t (\phi_i(s)) =
\phi_i( s^{\prime })$ with $s^{\prime} = \arctan (\tan(s) e^{t \beta(H)})$.
Hence, we conclude that the gradient flow leaves the path $\phi_i$ invariant.

Observe that $\beta$ is a positive root. So $e^{t \beta(H)}X_\beta \to 0$ as 
$t \to -\infty$ and hence $s^{\prime }\to 0$, i.e.,

\begin{center}
$\lim_{t \to - \infty} h^t \phi_i (s) = \phi_i(0) = wb_0 .$
\end{center}

When $t\rightarrow +\infty $ it follows that $\tan (s)e^{t\beta (H)}X_{\beta
}\rightarrow \pm \infty $ depending only in the sign of $\tan (s)$. Hence $%
s^{\prime }\rightarrow \pm \pi/2$, i.e.,

\begin{center}
$\lim_{t \to + \infty} h^t \phi_i (s) = \phi_i(\pm \pi/2) =
w^{\prime }b_0.$
\end{center}

Thus we get the desired result.

\begin{prop}
$\phi_i(s)$ give the two gradient flow lines between $wb_0$ and $w^{\prime }b_0$%
. One of them belongs to the interval $s \in (-\pi/2, 0)$ while the
other belongs to the interval $s \in (0, \pi/2)$.
\end{prop}

\section{Homology of maximal flag manifolds\label{sechommax}}

The cellular homology of a CW complex is defined from a cellular
decomposition of the complex and is isomorphic to the singular homology of
the space. It means that the homology group does not depend on the choice of
the cellular decomposition, although the boundary operator may change
according to the choice of the cellular decomposition, i.e., the way the
cells are glued.

In view of that, we \textbf{fix} once and for all reduced expressions 
\begin{equation*}
w=r_{1}\cdots r_{n}
\end{equation*}%
as a product of simple reflections, for each $w\in \mathcal{W}$. After
making these choices we define as before the characteristic map $\Phi _{w} $%
, $w\in \mathcal{W}$, that glues the ball $B_{w}$ in the union of Schubert
cells $\mathcal{S}_{u}$ with $u<w$.

\subsection{The boundary map}

We recall (in our context) the definition of the cellular
boundary maps giving the homology with coefficients in a ring $R$ (see \cite%
{Hat02}). Let $\mathcal{C}$ be the $R$-module freely generated by $\mathcal{S}%
_{w}$, $w\in \mathcal{W}$. The boundary maps $\partial :\mathcal{C}%
\rightarrow \mathcal{C}$ are defined by 
\begin{equation*}
\partial \mathcal{S}_{w}=\sum_{w^{\prime }}c(w,w^{\prime })\mathcal{S}%
_{w^{\prime }}
\end{equation*}%
where the coefficients $c(w,w^{\prime })\in R$ \ satisfy the properties:

\begin{enumerate}
\item $c(w,w^{\prime })=0$ in case $\dim \mathcal{S}_{w}-\dim \mathcal{S}%
_{w^{\prime }}\neq 1$.

\item If $\dim \mathcal{S}_{w}-\dim \mathcal{S}_{w^{\prime }}=1$ then $%
c(w,w^{\prime })=\deg \left( \phi _{w,w^{\prime }}:S_{w}^{d-1}\rightarrow
S_{w^{\prime }}^{d-1}\right) $, where $\phi _{w,w^{\prime }}$ is the
composition of the following maps:

\begin{enumerate}
\item The attaching map: $\Phi _{w}|_{\partial
}(B_{w}^{d}):S_{w}^{d-1}=\partial (B_{w}^{d})\rightarrow \mathcal{S}%
_{w}\setminus N\cdot wb_{0}=\cup _{u<w}\mathcal{S}_{u}=X^{d-1}$, where $%
X^{d-1}$ denotes the $(d-1)$-squeleton of $\mathcal{S}_{w}$.

\item The quotient map $X^{d-1}\rightarrow X^{d-1}/(X^{d-1}\setminus 
\mathcal{S}_{w^{\prime }})$ where we take the cell $\mathcal{S}_{w^{\prime
}} $ inside $X^{d-1}$ and identify its complement in $\mathcal{S}_{w^{\prime
}}$ to a point.

\item The identification: $X^{d-1}/(X^{d-1}\setminus \mathcal{S}_{w^{\prime
}})\cong \mathcal{S}_{w^{\prime }}/(\mathcal{S}_{w^{\prime }}\setminus
N\cdot w^{\prime }b_{0})$ which are in the same space. This last one is $%
\mathcal{S}_{w^{\prime }}/(\mathcal{S}_{w^{\prime }}\setminus N\cdot
w^{\prime }b_{0})=\sigma _{w^{\prime }}$ by definition.

\item $\Phi _{w^{\prime }}^{-1}:\sigma _{w^{\prime }}\rightarrow
S_{w^{\prime }}^{d-1}$. This is the map defined in (\ref{Sphere}).
\end{enumerate}
\end{enumerate}

\vspace{12pt}
\noindent
\textbf{Remark: } There is a subtlety which must be emphasized: $\phi
_{w,w^{\prime }}$ is a map $S^{d-1}\rightarrow S^{d-1}$ whose domain is the
boundary of a ball in some $\mathbb{R}^{N}$ (the ball $B_{w}$) and, hence,
it is a canonically defined sphere. However, the codomain is the space $%
\sigma _{w^{\prime }}$ which is homeomorphic to $S^{d-1}$. To get the
boundary map a homeomorphism $\sigma _{w^{\prime }}\rightarrow S^{d-1}$ must
be fixed beforehand, since distinct homeomorphisms may yield maps with
distinct degrees. Here is where it is needed to choose in advance the
reduced expressions of $w\in \mathcal{W}$.

To compute the the degree $c(w,w^{\prime })=\deg \left( \phi _{w,w^{\prime
}}:S_{w}^{d-1}\rightarrow S_{w^{\prime }}^{d-1}\right) $ when $w=r_{1}\cdots
r_{n}$ and $w^{\prime }=r_{1}\cdots \hat{r_{i}}\cdots r_{n}$ are minimal
decompositions we proceed with the following steps.

\subsubsection*{Step 1: Domain and codomain spheres}

First we identify the spheres $S_{w}^{d-1}$ in the domain and $S_{w^{\prime
}}^{d-1}$ in the codomain.

Remember that $B_{w}=B^{d_{1}}\times \cdots \times B^{d_{n}}$ where $%
B^{d_{i}}$ is the $1$-dimensional choosen to be the interval $[0,\pi ]$, as
in the construction of Lemma \ref{compact_parametrization_1}. The dimension
of $B_{w}$ is $d=d_{1}+\cdots +d_{n}$ and the domain of $\phi _{w,w^{\prime
}}$ is 
\begin{equation*}
S_{w}^{d-1}=\partial (B_{w})=\{(t_{1},\ldots ,t_{n}):\exists j,t_{j}\in
\partial B^{d_{j}}\}
\end{equation*}%
the union of \textquotedblleft faces" of $B_{w}$.

On the other hand let $B_{w^{\prime }}=B^{d_{1}}\times \cdots \times \hat{%
B^{d_{i}}}\times \cdots \times B^{d_{n}}$. Then codomain is the sphere $%
S_{w^{\prime }}^{d-1}$ obtained by collapsing to a point the boundary of $%
B_{w^{\prime }}$. This is seen by items (c) and (d) in the above definition
of $\partial $.

\subsubsection*{Step 2: $\protect\sigma _{w^{\prime }}$ in the image $\Phi
_{w}(S_{w}^{d-1})$.}

The second step is to see how $\sigma _{w^{\prime }}$ sits inside the image $%
\Phi _{w}(S_{w}^{d-1})$. The following lemma says how is the pre-image of $%
N\cdot w^{\prime }b_{0}$ under $\Phi _{w}$.

\begin{lema}
$\Phi _{w}(t_{1},\ldots ,t_{n})\in N\cdot w^{\prime }b_{0}$ if and only if $%
t_{j}\in (B^{d_{j}})^{\circ }$, $j\neq i$ and $t_{i}\in \partial B^{d_{i}}$,
that is, $t_{i}=0$ or $\pi $.
\end{lema}

\begin{proof}
If $t_{i}\in \partial B^{d_{i}}$ then $\psi _{i}(t_{i})\in M$ by Lemma \ref%
{compact_parametrization_1}. This implies that 
\begin{equation*}
\Phi _{w}(t_{1},\ldots ,t_{n})=\psi _{1}(t_{1})\cdots \psi _{n}(t_{n})\cdot
b_{0}\in K_{1}\cdots \hat{K_{i}}\cdots K_{n}=\mathcal{S}_{w^{\prime }}
\end{equation*}
since $M\subset K_{s}$ for all sub-index $s$. By Proposition \ref%
{propSchubert_3}, we see that $\Phi _{w}(t_{1},\ldots ,t_{i},\ldots
,t_{n})\in N\cdot w^{\prime }b_{0}$ if and only if $\psi _{j}(t_{j})\notin M$
for $j\neq i$, which in turn is equivalent to $t_{j}\in (B^{d_{j}})^{\circ } 
$, $i\neq j$, by Lemma \ref{compact_parametrization_1}.
\end{proof}

In other words the pre-image $\Phi _{w}^{-1}\left( N\cdot w^{\prime
}b_{0}\right) \subset B_{w}$ is the union of the interior of the two faces
corresponding to the $i$-th coordinate, that is, the faces where $t_{i}=0$
and $t_{i}=\pi $, respectively.

In the quotient $\sigma _{w^{\prime }}=\mathcal{S}_{w^{\prime }}/(\mathcal{S}%
_{w^{\prime }}\setminus N\cdot w^{\prime }b_{0})$ the faces of $\partial
B_{w}$ corresponding to the $j$-th coordinates, $j\neq i$ are collapsed to a
point.

\subsubsection*{Step 3: Degrees}

The degree of $\phi _{w,w^{\prime }}$ is the sum of the degree of two maps,
namely the maps obtaining by restricting to the faces%
\begin{equation*}
\mathcal{F}_{0}^{i}=\{(t_{1},\ldots ,0,\ldots ,t_{n})\}\quad \mathrm{and}%
\quad \mathcal{F}_{\pi }^{i}=\{(t_{1},\ldots ,\pi ,\ldots ,t_{n})\}.
\end{equation*}%
The values of $\phi _{w,w^{\prime }}$ in these faces are given by 
\begin{eqnarray*}
f_{i}^{0}\left( \mathbf{t}\right) &=&\Phi _{w^{\prime }}^{-1}\left( \psi
_{1}(t_{1})\cdots \psi _{i}(0)\cdots \psi _{n}(t_{n})\cdot b_{0}\right) \\
&=&\Phi _{w^{\prime }}^{-1}\left( \psi _{1}(t_{1})\cdots 1\cdots \psi
_{n}(t_{n})\cdot b_{0}\right) . \\
f_{i}^{\pi }\left( \mathbf{t}\right) &=&\Phi _{w^{\prime }}^{-1}\left( \psi
_{1}(t_{1})\cdots \psi _{i}(\pi )\cdots \psi _{n}(t_{n})\cdot b_{0}\right) \\
&=&\Phi _{w^{\prime }}^{-1}\left( \psi _{1}(t_{1})\cdots m_{\alpha
_{i}}\cdots \psi _{n}(t_{n})\cdot b_{0}\right) .
\end{eqnarray*}%
where $\mathbf{t}=\left( t_{1},\ldots ,\widehat{t_{i}},\ldots ,t_{n}\right) $
and $\Phi _{w^{\prime }}$ is given by a choice of a reduced expression $%
w^{\prime }=s_{1}\cdots s_{m}$ (choosed in advance) which may be different
from the reduced expression $w^{\prime }=r_{1}\cdots \hat{r_{i}}\cdots r_{n}$%
.

The degree of $\phi _{w,w^{\prime }}$ is the sum of the degrees of $%
f_{i}^{0} $ and $f_{i}^{\pi }$ which may be considered as maps $%
S^{d-1}\rightarrow S^{d-1}$ by collapsing the boundary to points of the
faces.

Now, the degree of a map $\varphi $ can be computed as the sum of the local
degrees in the inverse image of $\varphi ^{-1}(\xi )$ which has a finite
number of points (see \cite{Hat02}, Proposition 2.30).

In the case of our map $\phi _{w,w^{\prime }}$, the maps $f_{i}^{0}$ and $%
f_{i}^{\pi }$ are homeomorphisms so that pre-image $\phi _{w,w^{\prime
}}^{-1}(\xi )$ of a generic point has two points. Namely a point $x_{1}$ in
the the face $\mathcal{F}_{0}^{i}$ and another one $x_{2}$ in the face $%
\mathcal{F}_{\pi }^{i}$. The local degree at $x_{1}$ is the degree of $%
f_{i}^{0}$ since $f_{i}^{0}$ is a homeomorphism. The same the local degree
at $x_{2}$ is the degree of $f_{i}^{\pi }$.

Finally the degrees of $f_{i}^{0}$ and $f_{i}^{\pi }$ are $\pm 1$ since they
are homeomorphisms.

\vspace{12pt}
\noindent
\textbf{Summarizing:} To get the degree of $\phi _{w,w^{\prime }}$ we must
restrict $\Phi _{w^{\prime }}^{-1}\circ \Phi _{w}$ to the faces $\mathcal{F}%
_{0}^{i}$ and $\mathcal{F}_{\pi }^{i}$ and view these faces as spheres (with
the boundaries collapsed to points). The sum of the degrees of these two
restrictions is the degree of $\phi _{w,w^{\prime }}$.

The restrictions of $\Phi _{w^{\prime }}^{-1}\circ \Phi _{w}$ to the faces $%
\mathcal{F}_{0}^{i}$ and $\mathcal{F}_{\pi }^{i}$ are homeomorphisms and
hence have degree $\pm 1$. It follows that the total degre of $\phi
_{w,w^{\prime }}$ is $0$ or $\pm 2$. This is one of the main results on the
homology of flag manifolds.

\begin{teo}
The coefficient $c(w,w^{\prime })=\deg (f_{i}^{0})+\deg (f_{i}^{\pi })=0$ or 
$\pm 2$, for any $w,w^{\prime }\in \mathcal{W}$.
\end{teo}

In particular, in the case of $\mathbb{Z}_{2}$ coefficients all boundary
maps vanish.

\begin{coro}
The homology of $\mathbb{F}$ over $\mathbb{Z}_{2}$ is a vector space of
dimension $|\mathcal{W}|$.
\end{coro}

\vspace{12pt}
\noindent
\textbf{Remark:} The above computations are particularly interesting when
the simple root $\alpha _{i}$ has multiplicity $\dim \mathfrak{g}_{\alpha
_{i}}=1$. If all the simple roots have multiplicity $\geq 2$ then the
boundary operator $\partial $ is identically zero and homology is freely
generated by the Schubert cells. This happens in the classical case of the
complex Lie algebras, where any root has (real) multiplicity two. An example
of a real Lie algebra where the simple roots have multiplicity $\geq 2$ is
the real form of $\mathfrak{sl}\left( n,\mathbb{C}\right) $ whose Satake
diagram is

\begin{center}
\begin{picture}(170,25)

\put(0,0){\setlength{\unitlength}{1pt}
\put(0,15){\circle*{6}}
\put(3,15){\line(1,0){20}}
\put(26,15){\circle{6}}
\put(29,15){\line(1,0){20}}
\put(52,15){\circle*{6}}
\put(55,15){\line(1,0){17}}

\put(72,15){\begin{picture}(24,3)
\put(7,0){\line(1,0){2}}
\put(11,0){\line(1,0){2}} 
\put(14,0){\line(1,0){2}}
\end{picture}}

\put(96,15){\line(1,0){17}}
\put(116,15){\circle*{6}}
\put(119,15){\line(1,0){20}}
\put(142,15){\circle{6}}
\put(145,15){\line(1,0){20}}
\put(168,15){\circle*{6}}
}
\end{picture}
\end{center}

\noindent
In this case the simple roots are complex and hence their multiplicities are 
$\geq 2$.

\subsection{Illustration\label{subsecsl3}}

In order to illustrate the above description of the boundary operator $%
\partial $ we consider here the maximal flag manifold $\mathbb{F}$ of the split real form of $\mathfrak{sl}\left( 3,\mathbb{R}\right)$. In
this case, the Weyl group is $S_{3}$, the permutation group in three
elements. The simple reflections are $(12)=r_{\alpha _{1,2}}=r_{1}$ and $%
(23)=r_{\alpha _{2,3}}=r_{2}$. Only $(13)$ has two reduced expressions: $%
(13)=(12)(23)(12)$ and $(13)=(23)(12)(23)$. We fix the following minimal
decompositions 
\begin{equation*}
1,(12),(23),(123)=(12)(23),(132)=(23)(12),(13)=(12)(23)(12).
\end{equation*}%
Let $A=E_{1,2}-E_{2,1}$ and $B=E_{2,3}-E_{3,2}$ be the matrices whose
exponentials provide parametrizations for the compact groups $K_{1}$ and $%
K_{2}$ respectively. With these choices, the characteristic maps are

\begin{enumerate}
\item $\Phi_1 (0) = b_0$.

\item $\Phi_{(12)}(t)= e^{tA} \cdot b_0$, $t \in [0,\pi]$.

\item $\Phi_{(23)}(t)= e^{tB} \cdot b_0$, $t \in [0,\pi]$.

\item $\Phi_{(123)}(t,s)=e^{tA}e^{sB}\cdot b_0$, $(t,s) \in [0,\pi]^2$.

\item $\Phi_{(132)}(t,s)=e^{tB}e^{sA}\cdot b_0$, $(t,s) \in [0,\pi]^2$.

\item $\Phi_{(13)}(t,s,z)=e^{tA}e^{sB}e^{zA}\cdot b_0$, $(t,s,z) \in
[0,\pi]^3$.
\end{enumerate}

Then we obtain expressions for $c(w,w^{\prime })$.

\begin{enumerate}
\item $c((12),1)=0$ and $c((23),1)=0$ since there is a unique $0$-cell.

\item $c((123),(12))=0$. Note that $(12)=(12)\widehat{(23)}$. So we need to
consider the degree of the two maps $f_{2}^{0},f_{2}^{\pi }:S^{1}\rightarrow
S^{1}$ defined by $f_{2}^{0}(t,0)=e^{tA}e^{0B}\cdot b_{0}=e^{tA}\cdot b_{0}$
and $f_{2}^{\pi }(t,\pi )=e^{tA}e^{\pi B}\cdot b_{0}=e^{tA}\cdot b_{0}$.
Their degrees are obtained by comparing the orientation in the respective
face of the boundary of the cube $[0,\pi ]^{2}$, which is $S^{1}$ oriented
counter-clockwise, with the orientation given by the attaching map $%
e^{tA}\cdot b_{0}$. 
Hence, the degree of $f_{2}^{0}$ is $1$ since as $t$ increases the curve $%
(t,0)$ and $e^{tA}\cdot b_{0}$ go in the same direction as $\Phi _{(12)}$.
On the other hand the degree of $f_{2}^{\pi }$ is $-1$ since as $t$
increases, the curve $(t,0)$ and the image $e^{tA}\cdot b_{0}$ are in
opposite directions. Hence $c((123),(12))=+1+(-1)=0$.

\item $c((123),(23))=-2$. Here $(23)=\widehat{(12)}(23)$. So we consider the
degree of the two maps $f_{1}^{0},f_{1}^{\pi }:S^{1}\rightarrow S^{1}$ given
by $f_{1}^{0}(0,s)=e^{0A}e^{sB}\cdot b_{0}=e^{sB}\cdot b_{0}$ and $%
f_{1}^{\pi }(\pi ,s)=e^{\pi A}e^{sB}\cdot b_{0}=\exp (s\mathrm{Ad}(e^{\pi
A})B)\cdot b_{0}=e^{-sB}\cdot b_{0}$ because $\mathrm{Ad}(e^{\pi A})B=-B$.
Since $e^{-sB}\cdot b_{0}=e^{(\pi -s)B}e^{\pi B}\cdot b_{0}=e^{(\pi
-s)B}\cdot b_{0}$, the function $f_{1}^{\pi }$ defined in $[0,\pi ]$ is
given by $f_{1}^{\pi }(0,s)=e^{(\pi -s)B}\cdot b_{0}$. 
Hence the degree of $f_{1}^{0}$ is $-1$ as above. The degree of $f_{1}^{\pi
} $ is also $-1$ since it is the degree of the function $s\mapsto \pi -s$.
Hence $c((123),(12))=(-1)+(-1)=-2$.

\item $c((132),(121))=-2$ and $c((132),(23))=0$, which can be seen the same
as above.

\item $c((13),(123))=0$. Note that $(123)=(12)(23)\widehat{(12)}$. So we
consider the maps $f_{3}^{0},f_{3}^{\pi }:S^{2}\rightarrow S^{2}$.

\begin{enumerate}
\item $f_{3}^{0}(t,s,0)=e^{tA}e^{sB}e^{0A}\cdot b_{0}=e^{tA}e^{sB}\cdot
b_{0} $, and we have 
\begin{equation*}
\deg f_{3}^{0}=-1.
\end{equation*}%
In fact, the boundary of the cube $[0,\pi ]^{3}$ is $S^{2}$ oriented with
the normal vector pointing outwards. The face $(t,s,0)$ (in this order) when
viewed in the domain is negatively orientated while in the codomain the
orientation agrees with the $S^{2}$ orientation. Hence the degree is $-1$.

\item $f_{3}^{\pi }(t,s,\pi )=e^{tA}e^{sB}e^{\pi A}\cdot
b_{0}=e^{tA}e^{sB}\cdot b_{0}$ with 
\begin{equation*}
\deg f_{3}^{\pi }=1.
\end{equation*}%
In this case the face $(t,s,\pi )$ agrees with the positive orientation.
\end{enumerate}

\item $c((13),(132))=0$. Note that $(132)=\widehat{(12)}(23)(12)$. So we
consider the maps $f_{1}^{0},f_{1}^{\pi }:S^{2}\rightarrow S^{2}$.

\begin{enumerate}
\item $f_{1}^{0}(0,t,s)=e^{0A}e^{tB}e^{sA}\cdot b_{0}=e^{tB}e^{sA}\cdot
b_{0} $ with 
\begin{equation*}
\deg f_{3}^{0}=-1.
\end{equation*}%
In this case the face $(0,t,s)$ (in this order) in the domain hasa negative
orientation while in the codomain the orientation agrees with the positive
one. Hence the degree is $-1$.

\item $f_{1}^{\pi }(\pi ,t,s)=e^{\pi A}e^{tB}e^{sA}\cdot b_{0}=\exp
(-tB)e^{sA}\cdot b_{0}$ since $Ad(e^{\pi A}B)=-B$ and $Ad(e^{\pi A}A)=A$. We
want to describe this map with a domain in $[0,\pi ]^{2}$. So, first $\exp
(-tB)e^{sA}\cdot b_{0}=\exp ((\pi -t)B)e^{\pi B}e^{sA}\cdot b_{0}$. Finally,
since $Ad(e^{\pi B}A)=-A$ we get 
\begin{equation*}
\exp (-sB)e^{sA}\cdot b_{0}=e^{(\pi -s)B}e^{(\pi -s)A}\cdot b_{0}.
\end{equation*}%
Hence the degree of $f_{1}^{\pi }$ is the degree of $(t,s)\mapsto (\pi
-t,\pi -s)$. This degree is $+1$ since it preserves the orientation.
\end{enumerate}
\end{enumerate}

Summarizing, the boundary operator is given by 
\begin{itemize}
\item $\partial_3 \mathcal{S}_{(13)}=0$;
\item $\partial_2 \mathcal{S}_{(123)} = -2 \mathcal{S}_{(23)} $ and $\partial_2
\mathcal{S}_{(132)}= -2\mathcal{S}_{(12)}$;
\item $\partial_1 \mathcal{S}_{(12)} = \partial \mathcal{S}_{(23)}=0$.
\end{itemize}  
Hence the integer homology groups are
\begin{itemize}
\item $H_3(\mathbb{F},\mathbb{Z}) = \mathbb{Z}$ generated by $\mathcal{S}%
_{(13)}$.

\item $H_2(\mathbb{F},\mathbb{Z}) = 0$ ($\ker \partial_2 =0$).

\item $H_1(\mathbb{F},\mathbb{Z}) = \mathbb{Z}_2 \oplus \mathbb{Z}_2$ ($\ker
\partial_1$ is $\mathbb{Z}\oplus\mathbb{Z}$ and the image of $\partial_2$ is 
$2\mathbb{Z}\cdot \mathcal{S}_{(12)} \oplus 2\mathbb{Z}\cdot \mathcal{S}%
_{(23)}$).
\end{itemize}

\subsection{Algebraic expressions for the degrees}

Here we compute the coefficients $c(w,w^{\prime })$ in terms of the roots by
finding the degrees of the maps involved.

For a diffeomorphism $\varphi $ of the sphere its degree is local degree at
a point $x$ which in turn is the sign of the determinant $\det (d\varphi
_{x})$ with respect to a volume form of $S^{d}$. Let us apply this in our
context.

We let $w=r_{1}\cdots r_{n}$ and $w^{\prime }=r_{1}\cdots \hat{r_{i}}\cdots
r_{n}$ be reduced expressions, with $r_{i}=r_{\alpha _{i}}$ and assume
throughout that the simple root $\alpha _{i}$ has multiplicity $%
d_{i}=d_{\alpha _{i}}=1$.

We must find the degrees of $f_{i}^{0}$ and $f_{i}^{\pi }$ defined by

\begin{enumerate}
\item $f_i^0 (t_1, \ldots, 0, \ldots, t_n) = \Phi_{w^{\prime }}^{-1}\left(
\psi_1(t_1) \cdots 1 \cdots \psi_n(t_n) \cdot b_0 \right).$

\item $f_i^\pi (t_1, \ldots, \pi , \ldots, t_n) = \Phi_{w^{\prime
}}^{-1}\left( \psi_1(t_1) \cdots m_{\alpha_i} \cdots \psi_n(t_n) \cdot b_0
\right).$
\end{enumerate}

In these expressions $\Phi _{w^{\prime }}^{-1}$ is defined by a previously
chosen reduced expression $w^{\prime }=s_{1}\cdots s_{m}$ of $w^{\prime }$
which may be distinct of $w^{\prime }=r_{1}\cdots \hat{r_{i}}\cdots r_{n}$.
On the other hand $w^{\prime }=r_{1}\cdots \hat{r_{i}}\cdots r_{n}$ can be
used to define another characteristic map, which will be denoted by $\Psi
_{w^{\prime }}$. This new characteristic map define new functions

\begin{enumerate}
\item $p_i^0 (t_1, \ldots, 0, \ldots, t_n) = \Psi_{w^{\prime }}^{-1}\left(
\psi_1(t_1) \cdots 1 \cdots \psi_n(t_n) \cdot b_0 \right)$ and

\item $p_{i}^{\pi }(t_{1},\ldots ,\pi ,\ldots ,t_{n})=\Psi _{w^{\prime
}}^{-1}\left( \psi _{1}(t_{1})\cdots m_{\alpha _{i}}\cdots \psi
_{n}(t_{n})\cdot b_{0}\right) $.
\end{enumerate}

The two pair of functions are related by 
\begin{equation*}
f_{i}^{\epsilon }=\left( \Phi _{w^{\prime }}^{-1}\circ \Psi _{w^{\prime
}}\right) \circ p_{i}^{\epsilon }\hspace{1cm}\epsilon =0,\pi .
\end{equation*}

The composition $\Phi _{w^{\prime }}^{-1}\circ \Psi _{w^{\prime }}$ (also
understood as a map between spheres in which the boundary are collapsed to
points) are homeomorphisms of spheres and, hence, have degree $\pm 1$. Hence
we can concentrate on the computation of degrees of the ${p_{i}^{\epsilon }} 
$'s since the total degree will be multiplied by $\pm 1$.

Before getting these degrees we make the following discussion on the
orientation of the faces of the cube $\left[ -1,1\right] ^{d}$, centered at
the origin of $\mathbb{R}^{d}$, which is given with the basis $%
\{e_{1},\ldots ,e_{d}\}$.

Starting with the $(d-1)$-dimensional sphere $S^{d-1}$ we orient the tangent
space at $x\in S^{d-1}$ by a basis $\{f_{2},\ldots ,f_{d}\}$ such that $%
\{x,f_{2},\ldots ,f_{d}\}$ is positively oriented. The faces of $[-1,1]^{d}$
are oriented accordingly: Given a base vector $e_{j}$, we let $F_{j}^{-}$ be
the face perpendicular to $e_{j}$ that contains $-e_{j}$ and $F_{j}^{+}$ the
one that contains $e_{j}$. Then $F_{j}^{-}$ has the same orientation as the
basis $e_{1},\ldots ,\hat{e_{j}},\ldots ,e_{d}$ if $j$ is even ($%
-e_{j},e_{1},\ldots ,\hat{e_{j}},\ldots ,e_{d}$ is positively oriented in $%
\mathbb{R}^{d}$) and opposite orientation if $j$ is odd. Therefore the
orientation of $F_{j}^{-}$ is $\left( -1\right) ^{j}$ the orientation of $%
e_{1},\ldots ,\hat{e_{j}},\ldots ,e_{d}$. Analogously, the orientation of $%
F_{j}^{+}$ is $\left( -1\right) ^{j+1}$ the orientation of $e_{1},\ldots ,%
\hat{e_{j}},\ldots ,e_{d}$.

The following facts about the action of an element $m\in M$ will be used
below in the computation of the degrees.

\begin{lema}
\label{lemproperties}For a root $\alpha $ consider the action on $\mathbb{F}$
of $m=m_{\alpha }=\exp (\pi A_{\alpha })\in M$. Then

\begin{enumerate}
\item $mwb_{0}=wb_{0}$ and $mNm^{-1}=N$. Therefore $m$ leaves invariant any
Bruhat cell and hence any Schubert cell $\mathcal{S}_{w}$.

\item The restriction of $m$ to $N\cdot wb_{0}$ is a diffeomorphism.

\item The differential $dm_{wb_{0}}$ identifies to $\mathrm{Ad}\left(
m\right) $ restricted to the subspace%
\begin{equation*}
\sum_{\beta \in \Pi_w}\mathfrak{g}_{\beta }.
\end{equation*}
\end{enumerate}
\end{lema}

\begin{proof} 
Since $\mathrm{Ad}(m_{\alpha })\mathfrak{g}_{\beta }=\mathfrak{g}_{\beta }$, $\beta \in \Pi $
(cf. Lemma \ref{compact_parametrization_2}), the first and second statements follow easily. 

For the third statement we use the notation $X\cdot x=d/dt\left( \exp tX\right) _{t=0}$, $x\in \mathbb{F}$ and $%
X\in \mathfrak{g}$. Also, for $A\subset \mathfrak{g}$ write $A\cdot
x=\{X\cdot x:X\in A\}$.

Note that $N\cdot wb_{0}=w(w^{-1}Nw)\cdot b_{0}$, and the tangent space
to $(w^{-1}Nw)\cdot b_{0}$ at $b_{0}$ is spanned by $\mathfrak{g}_{\alpha
}\cdot b_{0}$ with $\alpha <0$ such that $\alpha =w^{-1}\beta $ and $\beta
>0 $, that is, $w\cdot \alpha >0$. Since $\left( dw\right) (\mathfrak{g}%
_{\alpha }\cdot b_{0})=\mathfrak{g}_{w\cdot \alpha }\cdot b_{0}$, it follows
that $T_{wb_{0}}\left( N\cdot wb_{0}\right) $ is spanned by $\mathfrak{g}%
_{\beta }\cdot b_{0}$ with $\beta =w\cdot \alpha >0$ such that $w^{-1}\cdot
\beta =\alpha <0$. Hence the result.
\end{proof}

The next statement computes the degree of ${p_{i}^{\epsilon }}$'s in terms
of Killing numbers.

\begin{prop}
\label{propdegresig}$\deg (p_{i}^{0})=(-1)^{I}$ and $\deg (p_{i}^{\pi
})=(-1)^{I+1+\sigma }$, where 
\begin{equation}
\sigma =\sigma \left( w,w^{\prime }\right) =\sum_{\beta \in \Pi _{u}}\frac{%
2\langle \alpha _{i},\beta \rangle }{\langle \alpha _{i},\alpha _{i}\rangle }%
\dim \mathfrak{g}_{\beta },\hspace{0.5cm}\Pi _{u}=\Pi ^{+}\cap u\Pi ^{-},%
\hspace{0.5cm}u=r_{i+1}\cdots r_{n},  \label{forsigmasoma}
\end{equation}%
and $I$ is the sum of the multiplicities of the roots $\alpha _{j}$ with $%
j\leq i$.
\end{prop}

\begin{proof}
The map $p_{i}^{0}$ is the projection of the face of a $d$-dimensional cube 
onto the face of a $(d-1)$-dimensional cube, i.e., in coordinates 
\[(t_{1},\ldots ,0,\ldots ,t_{n}) \mapsto (t_{1},\ldots ,\hat{t_{i}},\ldots ,t_{n}).\]
Note that with respect to the basis $e_{1},\ldots ,e_{d}$ the $t_{i}$-coordinate appears in
the $I$-th position. Hence, by the orientation of the cube, discussed above,
the projection preserves or reverses orientation if $I$ is even or odd,
respectively. Therefore, $\deg (p_{i}^{0})=(-1)^{I}$.

To get $\deg (p_{i}^{\pi })$ write $m_{i}=m_{\alpha _{i}}$ for the element
of $M$ appearing in the expression of $p_{i}^{\pi }$. Its action on $\mathbb{%
F}$ leaves invariant any Bruhat cell $N\cdot wb_{0}$ (because $%
m_{i}Nm_{i}^{-1}=N$ and $m_{i}wb_{0}=wb_{0}$), and hence any Schubert cell.
Moreover, the restriction of $m_{i}$ to $N\cdot wb_{0}$ is a diffeomorphism
(given by the conjugation $y\in N\mapsto m_{i}ym_{i}^{-1}$).

In particular, we restrict the action of $m_{i}$ to the cell $\mathcal{S}%
_{u} $, $u=r_{i+1}\cdots r_{n}$. Using the parametrization of this cell
by the cube $B_{u}$ we get 
\begin{equation*}
m_{i}\psi _{i+1}(t_{i+1})\cdots \psi _{n}(t_{n})\cdot b_{0}=\psi
_{i+1}(s_{i+1})\cdots \psi _{n}(s_{n})\cdot b_{0},
\end{equation*}%
with $(s_{i+1},\ldots ,s_{n})=\overline{m}_{i}(t_{i+1},\ldots ,t_{n})$ with $%
\overline{m}_{i}:B_{u}\rightarrow B_{u}$ continuous and a diffeomorphism of
the interior of $B_{u}$.

Hence, $p_{i}^{\pi }(t_{1},\ldots ,\pi ,\ldots ,t_{n})$ becomes the
projection of the $i-1$ first coordinates and the composition of $\overline{m%
}_{i}$ with the projection of the last $j$-coordinates, $j=i+1,\ldots ,n$.
From the choice of the orientation of $B_{w}=\left[ 0,\pi \right] ^{d}$, the
face $(t_{1},\ldots ,\pi ,\ldots ,t_{n})$ of $B_{w}$ has orientation $%
(-1)^{I+1}$ with respect to the orientation of the coordinates $%
(t_{1},\ldots ,\widehat{t_{i}},\ldots ,t_{n})$. Hence, after collapsing the
boundary to a point, we get the degree 
\begin{equation*}
\deg p_{i}^{\pi }=(-1)^{I+1}\deg \overline{m}_{i}.
\end{equation*}%
The degree of $\overline{m}_{i}$ equals its local degree at one point which
in turn is sign of the determinant of the differential $d({m}_{i})_{ub_{0}}$
restriced to the tangent space to Bruhat cell $N\cdot ub_{0}$ at $ub_{0}$: 
\begin{equation*}
\deg (p_{i}^{\pi })=(-1)^{I+1}\mathrm{sgn}[\det \left( d({m}%
_{i})_{ub_{0}}|T_{ub_{0}}(N\cdot ub_{0})\right) ].
\end{equation*}%
By the third statement in the Lemma \ref{lemproperties}, $T_{ub_{0}}(N\cdot ub_{0})$
identifies to $\sum_{\beta \in \Pi_w}\mathfrak{g}_{\beta }$.

Once we have the generators $\mathfrak{g}_{\beta }\cdot ub_{0}$, $\beta \in
\Pi_w$ for $T_{ub_{0}}(N\cdot ub_{0})$ together with the action of $\mathrm{%
Ad}(m_{i})$ over $\mathfrak{g}_{\beta }$ given by the Lemma \ref%
{compact_parametrization_2}, $\mathrm{Ad}(m_{\alpha })_{\left\vert \mathfrak{%
g}_{\beta }\right. }=(-1)^{\epsilon (\alpha ,\beta )}\mathrm{id}$ we
conclude that the signal of $\det \left( d({m}_{i})_{ub_{0}}|T_{ub_{0}}(N%
\cdot ub_{0})\right) =(-1)^{\sigma }$ where 
\begin{equation*}
\sigma =\sum_{\beta \in \Pi _{u}}\frac{2\langle \alpha _{i},\beta \rangle }{%
\langle \alpha _{i},\alpha _{i}\rangle }\dim \mathfrak{g}_{\beta }.
\qedhere
\end{equation*}
\end{proof}

Summarizing, we have the following algebraic expression for the coefficient $%
c\left( w,w^{\prime }\right) $.

\begin{teo}
\label{teoforcw}Let be $\sigma \left( w,w^{\prime }\right) $ be defined as
in (\ref{forsigmasoma}). Then 
\begin{equation*}
c(w,w^{\prime })=\deg \left( \Phi _{w^{\prime }}^{-1}\circ \Psi _{w^{\prime
}}\right) (-1)^{I}(1-(-1)^{\sigma \left( w,w^{\prime }\right) }).
\end{equation*}
\end{teo}

We will now derive another formula for $\sigma \left( w,w^{\prime }\right) $
that does not depend on the reduced expressions of $w$ and $w^{\prime }$.
This formula is the same one given by Theorem A of \cite{Koc95}.

For $w\in \mathcal{W}$, let 
\begin{equation*}
\phi (w)=\sum_{\beta \in \Pi _{w}}\dim \mathfrak{g}_{\beta }\cdot \beta
\end{equation*}%
be the sum of roots in $\Pi _{w}=\Pi ^{+}\cap w\Pi ^{-}$ counted with their
multiplicity.

As before let $w=r_{1}\cdots r_{n}$ and $w^{\prime }=r_{1}\cdots \widehat{%
r_{i}}\cdots r_{n}$ be reduced expressions.

\begin{prop}
Let $\beta $ be the unique root (not necessarily simple) such that $%
w=r_{\beta }w^{\prime }$, that is, $\beta =r_{1}\cdots r_{i-1}\alpha _{i}$.
Then 
\begin{equation*}
\phi (w)-\phi (w^{\prime })=(1-\sigma )\beta
\end{equation*}%
where $\sigma =\sigma (w,w^{\prime })$ is the sum (\ref{forsigmasoma}).
\end{prop}

\begin{proof}
By the reduced expressions $w^{-1}=r_{n}\cdots r_{1}$ and $w^{\prime
-1}=r_{n}\cdots \hat{r_{i}}\cdots r_{1}$ and $u^{-1}=r_{n}\cdots r_{i+1}$ we
obtain the sets

\begin{enumerate}
\item $\Pi _{w}=\{\alpha _{1},r_{1}\alpha _{2},\ldots ,r_{1}\cdots r_{i-1}\alpha _{i},r_{1}\cdots r_{i}\alpha
_{i+1},\ldots ,r_{1}\cdots r_{n-1}\alpha _{n}\}$.

\item $\Pi _{w^{\prime }}=\{\alpha _{1},r_{1}\alpha _{2},\ldots , r_{1}\cdots r_{i-1}\alpha _{i+1},\ldots ,r_{1}\cdots 
\hat{r_{i}}\cdots r_{n-1}\alpha _{n}\}$.

\item $\Pi _{u}=\{\alpha _{i+1},r_{i+1}\alpha _{i+2},\ldots ,r_{i+1}\cdots
r_{n-1}\alpha _{n}\}$.
\end{enumerate}

The first $(i-1)$ roots of $\Pi _{w}$ and $\Pi _{w^{\prime }}$ coincide. The
remaining ones are related by the equalities 
\begin{equation*}
(r_{1}\cdots r_{i-1})r_{i}\cdots r_{j}\alpha _{j+1}=r_{\beta }(r_{1}\cdots
r_{i-1})r_{i+1}\cdots r_{j}\alpha _{j+1} \quad,\quad j=i,\ldots ,n-1,
\end{equation*}%
because $(r_{1}\cdots r_{i-1})r_{i}(r_{1}\cdots r_{i-1})^{-1}=r_{r_{1}\cdots
r_{i-1}\alpha _{i}}=r_{\beta }$. It follows that the remaining roots $r_{1}\cdots
r_{j}\alpha _{j+1}$ and the roots $r_{1}\cdots \hat{r_{i}}\cdots r_{j}\alpha
_{j+1}$ have the same multiplicity $d_{j}$, $j=i,\ldots ,n-1$. Write $\gamma
_{j}=r_{i+1}\cdots r_{j}\alpha _{j+1}$, so that $\Pi _{u}=\{\gamma
_{i},\gamma _{i+1},\ldots ,\gamma _{n-1}\}$. Then%
\begin{equation}
\phi (w)-\phi (w^{\prime })=\beta +\sum_{j=i}^{n-1}d_{j}(r_{1}\cdots
r_{i-1})\left( r_{i}(\gamma _{j})-\gamma _{j}\right)  \label{root_equation}
\end{equation}%
because $\beta =r_{1}\cdots r_{i-1}\alpha _{i}$ has multiplicity $1$ as $%
\alpha _{i}$.

Since $r_{i}(\gamma _{j})-\gamma _{j}=-\displaystyle\frac{2\langle \alpha
_{i},\gamma _{j}\rangle }{\langle \alpha _{i},\alpha _{i}\rangle }\alpha
_{i} $ we rewrite (\ref{root_equation}) as 
\begin{eqnarray}
\phi (w)-\phi (w^{\prime }) &=&\left( 1-\sum_{j=i}^{n-1}d_{j}\frac{2\langle
\alpha _{i},\gamma _{j}\rangle }{\langle \alpha _{i},\alpha _{i}\rangle }%
\right) \beta  \label{root_equation_2} \\
&=&(1-\sigma )\beta  \notag
\end{eqnarray}
concluding the proof.
\end{proof}

Combining the above proposition with Theorem \ref{teoforcw} we get
immediately the following formula for $c\left( w,w^{\prime }\right) $ (cf. 
\cite{Koc95}, Theorem A).

\begin{teo}
\label{teoforcw1}%
\begin{equation}
c(w,w^{\prime })=\deg \left( \Phi _{w^{\prime }}^{-1}\circ \Psi _{w^{\prime
}}\right) (-1)^{I}(1+(-1)^{\kappa(w,w')})  \label{forcewseg}
\end{equation}%
where $\kappa(w,w')$ is the integer defined by $\phi (w)-\phi (w^{\prime })=\kappa(w,w')\cdot
\beta $ and $\beta $ is the unique root such that $w=r_{\beta }w^{\prime }$.
\end{teo}

\vspace{12pt}

\noindent
\textbf{Remark:} If $w=r_{1}\cdots r_{n}$ and $w^{\prime }=r_{1}\cdots
r_{i-1}$ then $c(w,w^{\prime })=0$ because $m_{\alpha _{n}}$ does not affect
the computations of the degrees (see Proposition \ref{propdegresig}).

\vspace{12pt}
\noindent
\textbf{Example of $\mathrm{Sl}(3,\mathbb{R})$:} Let us use Formula (\ref{forcewseg}) to rederive the
homology of the maximal flag of $\mathrm{Sl}(3,\mathbb{R})$, the split real form of the algebra whose  Dynkin diagram is $A_2$. Let fix the
same reduced expressions for elements in $\mathcal{W}$ and note that the
unique element which has more than one reduced expression is $(13)$ which
implies that the factor $\left( \Phi _{w^{\prime }}^{-1}\circ \Psi
_{w^{\prime }}\right) $ is $1$ in all cases. In this case, we have the
following table \ref{table1} which determines completely the coefficients $c(w,w^{\prime })$, as
in Subsection \ref{subsecsl3}.

\begin{table}[h]\label{table1}
  \centering
    \begin{tabular*}{0.4 \textwidth}{||@{\extracolsep{\fill}}c|c|c||}
\hline
\multicolumn{3}{|c|}{Maximal Flag of $A_2$} \\ \hline
${\mathcal{W}}$ & ${\Pi _{w}}$ & ${\phi (w)}$ \\ \hline
~$1$ & $\emptyset $ & $0$ \\ \hline
~$(12)$ & $\alpha _{1}$ & $\alpha _{1}$ \\ \hline
~$(23)$ & $\alpha _{2}$ & $\alpha _{2}$ \\ \hline
~$(123)$ & $\alpha _{1},\alpha _{1}+\alpha _{2}$ & $2\alpha _{1}+\alpha _{2}$
\\ \hline
~$(132)$ & $\alpha _{2},\alpha _{1}+\alpha _{2}$ & $\alpha _{1}+2\alpha _{2}$
\\ \hline
~$(13)$ & $\Pi ^{+}$ & $2\alpha _{1}+2\alpha _{2}$ \\ \hline
\end{tabular*}
\vspace{12pt}

  \caption{Homology of the Maximal Flag of $A_2$}
\end{table}

For instance, let us compute $\partial_3 \mathcal{S}_{(13)}=0$. According to the table \ref{table1}, $\sigma((13),(123))=2(\alpha_1+\alpha_2)-(2\alpha_1+\alpha_2)= \alpha_2$ and $\sigma((13),(132))= 2(\alpha_1+\alpha_2)-(\alpha_1+2\alpha_2)=\alpha_1$. It implies that $\kappa((13),(123))=\kappa((13),(132))=1$ by which we conclude that $c((13),(123))=c((13),(132))=0$.

\vspace{12pt}
\noindent

\textbf{Example of $G_2$:} Let us apply the results above for the two groups with Dynkin diagram $G_{2}$, namely the complex group and the split real form. In the complex case, we already have $\partial =0$. Now we proceed to the real case. Let $\Sigma =\{\alpha _{1},\alpha _{2}\}$ be the simple roots. The set $\Pi ^{+}\setminus \Sigma =\{\alpha _{3}=\alpha
_{2}+\alpha _{1},\alpha _{4}=\alpha _{1}+2\alpha _{2},\alpha _{5}=\alpha
_{1}+3\alpha _{2},\alpha _{6}=2\alpha _{1}+3\alpha _{2}\}$ contains the remaining positive roots. 
The Weyl group with the respective fixed reduced expressions is $\mathcal{W}%
=\{1,r_{1},r_{2},s_{1}=r_{1}r_{2},s_{2}=r_{2}r_{1},r_{1}s_{2},r_{2}s_{1},s_{1}^{2},s_{2}^{2},r_{1}s_{2}^{2},r_{2}s_{1}^{2},s_{1}^{3}\}$, where $r_{i}=r_{\alpha _{i}}$ are the simple reflections and $%
s_{1}^{3}=s_{2}^{3}$ is the unique element with two different minimal
decompositions. The next table \ref{table2} presents the data useful to compute the homology coefficients.

\begin{table}[h]\label{table2}
  \centering
\begin{tabular*}{0.45 \textwidth}{||@{\extracolsep{\fill}}c|c|c||}
\hline
\multicolumn{3}{|c|}{Maximal flag of $G_{2}$} \\ \hline
${\mathcal{W}}$ & ${{\Pi _{w}}}$ & ${\phi (w)}$ \\ \hline
~$1$ & $\emptyset $ & $0$ \\ \hline
~$r_{1}$ & $\alpha _{1}$ & $\alpha _{1}$ \\ \hline
~$r_{2}$ & $\alpha _{2}$ & $\alpha _{2}$ \\ \hline
~$s_{1}$ & $\alpha _{1},\alpha _{3}$ & $2\alpha _{1}+\alpha _{2}$ \\ \hline
~$s_{2}$ & $\alpha _{2},\alpha _{5}$ & $\alpha _{1}+4\alpha _{2}$ \\ \hline
~$r_{1}s_{2}$ & $\alpha _{1},\alpha _{3},\alpha _{6}$ & $4\alpha
_{1}+4\alpha_{2}$ \\ \hline
~$r_{2}s_{1}$ & $\alpha _{2},\alpha _{5},\alpha _{4}$ & $2\alpha
_{1}+6\alpha_{2}$ \\ \hline
~$s_{1}^{2}$ & $\alpha _{1},\alpha _{3},\alpha _{6},\alpha _{4}$ & $%
5\alpha_{1}+6\alpha _{2}$ \\ \hline
~$s_{2}^{2}$ & $\alpha _{2},\alpha _{5},\alpha _{4},\alpha _{6}$ & $%
4\alpha_{1}+9\alpha _{2}$ \\ \hline
~$r_{1}s_{2}^{2}$ & $\alpha _{1},\alpha _{3},\alpha _{6},\alpha
_{4},\alpha_{5}$ & $6\alpha _{1}+9\alpha _{2}$ \\ \hline
~$r_{2}s_{1}^{2}$ & $\alpha _{2},\alpha _{5},\alpha _{4},\alpha
_{6},\alpha_{3}$ & $5\alpha _{1}+10\alpha _{2}$ \\ \hline
~$s_{1}^{3}$ & $\Pi ^{+}$ & $6\alpha _{1}+10\alpha _{2}$ \\ \hline
\end{tabular*}
\vspace{12pt}
  \caption{Homology of the maximal flag of $G_{2}$}
\end{table}

By (\ref{forcewseg}) the boundary operator is given as
\begin{itemize}
\item $\partial_6 \mathcal{S}_{s_{1}^{3}}=0$;
\item  $\partial_5 \mathcal{S}_{r_{1}s_{2}^{2}}=-2\mathcal{S}%
_{s_{2}^{2}}$ and $\partial_5 \mathcal{S}_{r_{2}s_{1}^{2}}=-2\mathcal{S}%
_{s_{1}^{2}}$;
\item $\partial_4 \mathcal{S}_{s_{1}^{2}}=\partial_4 \mathcal{S}_{s_{2}^{2}}=0$;
\item $\partial_3\mathcal{S}_{r_{1}s_{2}}=\partial_3 \mathcal{S}_{r_{2}s_{1}}=0$;
\item $\partial_2 \mathcal{S}_{s_{1}}=-2\mathcal{S}_{r_{2}}$ and $\partial_2
\mathcal{S}_{s_{2}}=-2\mathcal{S}_{r_{1}}$; 
\item $\partial_1 \mathcal{S}_{r_{1}}=\partial_1 \mathcal{S}_{r_{2}}=0$.
\end{itemize}
Hence 
\begin{itemize}
\item $H_6(\mathbb{F},\mathbb{Z})= \mathbb{Z}$. 

\item $H_5(\mathbb{F},\mathbb{Z})= 0$. 

\item $H_4(\mathbb{F},\mathbb{Z}) = \mathbb{Z}_2 \oplus \mathbb{Z}_2$. 

\item $H_3(\mathbb{F},\mathbb{Z}) = \mathbb{Z} \oplus \mathbb{Z}$. 

\item $H_2(\mathbb{F},\mathbb{Z}) = 0$. 

\item $H_1(\mathbb{F},\mathbb{Z}) = \mathbb{Z}_2 \oplus \mathbb{Z}_2$. 
\end{itemize}

\section{Partial flag manifolds\label{secpartial}}

In this section we project down the constructions made for the maximal flag
manifolds, via the canonical map $\pi _{\Theta }:\mathbb{F}\rightarrow 
\mathbb{F}_{\Theta }$, to obtain analogous results for the homology of a
partial flag manifold. In $\mathbb{F}_{\Theta }$ the Schubert cells are $\mathcal{S}_{w}^{\Theta }$%
, $w\in \mathcal{W}/\mathcal{W}_{\Theta }$, with $\mathcal{S}_{w}^{\Theta }=%
\mathcal{S}_{w_{1}}^{\Theta }$ if $w\mathcal{W}_{\Theta }={w_{1}}\mathcal{W}%
_{\Theta }$. The next lemma chooses a special representative in $w\mathcal{W}_{\Theta }$
for $\mathcal{S}_{w}^{\Theta}$.

\begin{lema}
\label{minimal_element} There exists an element $w_{1}=wu$ of the coset $w%
\mathcal{W}_{\Theta }$ such that 
\begin{equation*}
\dim \mathcal{S}_{w}^{\Theta }=\dim \mathcal{S}_{w_{1}}.
\end{equation*}%
This element is unique and minimal with respect to the Bruhat-Chevalley
order.
\end{lema}

\begin{proof}
By Proposition 1.1.2.13 of \cite{War72} any $v\in \mathcal{W}$ can be written
uniquely as 
\begin{equation*}
v=v_{s}v_{u}
\end{equation*}%
with $v_{s}\in \mathcal{W}_{\Theta }$ and $v_{u}$ satisfying $\Pi ^{+}\cap
v_{u}\Pi ^{-}\cap \langle \Theta \rangle =\emptyset $, that is, no positive
root in $\langle \Theta \rangle $ is mapped to a negative root by $%
v_{u}^{-1} $. Note that the condition for $v_{u}$ is equivalent to $\Pi
^{-}\cap v_{u}\Pi ^{+}\cap \langle \Theta \rangle =\emptyset $, since a root 
$\alpha >0$ belongs to $\Pi ^{+}\cap v_{u}\Pi ^{-}\cap \langle \Theta
\rangle $ if and only if $-\alpha \in \Pi ^{-}\cap v_{u}\Pi ^{+}\cap \langle
\Theta \rangle $.

Let $w^{-1}=v_{s}v_{u}$ be the decomposition for $w^{-1}$ so that $%
w=v_{u}^{-1}v_{s}^{-1}$. Then $w_{1}=v_{u}^{-1}\in w\mathcal{W}_{\Theta }$
is the required element.

In fact $\Pi ^{-}\cap w_{1}^{-1}\Pi ^{+}\cap \langle \Theta \rangle
=\emptyset $, and hence $\Pi ^{+}\cap w_{1}\Pi ^{-}\cap w_{1}\langle \Theta
\rangle =\emptyset $.

Now the tangent space $T_{w_{1}b_{0}}(N\cdot w_{1}b_{0})$ is 
\begin{equation*}
\langle \mathfrak{g}_{\beta }\cdot w_{1}b_{0}\,:\beta \in \Pi ^{+}\cap
w_{1}\Pi ^{-}\rangle
\end{equation*}%
(cf. Lemma \ref{lemproperties}). On the other hand the tangent space to the
fiber $\pi _{\Theta }^{-1}(w_{1}b_{\Theta })$ is the translation under $%
w_{1} $ of the tangent space to fiber at origin. Hence, $T_{w_{1}b_{\Theta
}}\pi _{\Theta }^{-1}(w_{1}b_{\Theta })$ is spanned by $w_{1}(\mathfrak{g}%
_{\alpha }\cdot b_{0})$, with $\alpha \in \langle \Theta \rangle $ and $%
\alpha <0$. By the translation formula, we have $w_{1}(\mathfrak{g}_{\alpha
}\cdot b_{0})=\mathfrak{g}_{w_{1}\alpha }\cdot w_{1}b_{0}$. Therefore, by
writting $\gamma =w_{1}\alpha $ we conclude that $T_{w_{1}b_{\Theta }}\pi
_{\Theta }^{-1}(w_{1}b_{\Theta })$ is spanned by $\mathfrak{g}_{\gamma
}\cdot w_{1}b_{0}$ with $w_{1}^{-1}\gamma \in \langle \Theta \rangle $ and $%
w^{-1}\gamma <0$, that is, with $w_{1}^{-1}\gamma \in \Pi ^{-}\cap \langle
\Theta \rangle $. So that 
\begin{equation*}
T_{w_{1}b_{0}}(\pi _{\Theta }^{-1}\left( w_{1}b_{\Theta }\right) )=\langle 
\mathfrak{g}_{\gamma }\cdot w_{1}b_{0}\,:\gamma \in w_{1}\Pi ^{-}\cap
w_{1}\langle \Theta \rangle \rangle .
\end{equation*}%
Since $\Pi ^{+}\cap w_{1}\Pi ^{-}\cap w_{1}\langle \Theta \rangle =\emptyset 
$, it follows that none of roots $\gamma $ spanning $T_{w_{1}b_{0}}(\pi
_{\Theta }^{-1}\left( w_{1}b_{\Theta }\right) )$ can be positive.

Therefore, $T_{w_{1}b_{0}}(N\cdot w_{1}b_{0})\cap T_{w_{1}b_{0}}(\pi
_{\Theta }^{-1}\left( w_{1}b_{\Theta }\right) )=\{0\}$. This implies that
the differential of $\pi _{\Theta }$ maps $T_{w_{1}b_{0}}(N\cdot w_{1}b_{0})$
isomorphically to the tangent space of $\pi _{\Theta }\left( N\cdot
w_{1}b_{0}\right) =N\cdot w_{1}b_{\Theta }$. Hence the two Bruhat cells have
the same dimension.

Finally, $N\cdot w_{1}b_{\Theta }$ has the minimum possible dimension among
the cells $N\cdot wb_{0}$, $w\in w_{1}\mathcal{W}$, because all of them
project onto $N\cdot w_{1}b_{\Theta }$. Hence $w_{1}$ has minimal length in $%
w_{1}\mathcal{W}$ which is known to be unique and minimal with respect to
the Bruhat-Chevalley order as well (see Deodhar \cite{Deo77}).
\end{proof}

We will denote by $\mathcal{W}_{\Theta }^{\min }$ the set of minimal
elements of the cosets in $\mathcal{W}/\mathcal{W}_{\Theta }$.

Now we contruct a cellular decomposition for $\mathbb{F}_{\Theta }$ with the
aid of the minimal elements $w\in \mathcal{W}_{\Theta }^{\min }$ in their
cosets $w\mathcal{W}_{\Theta }$, satisfying $\dim \mathcal{S}_{w}^{\Theta
}=\dim \mathcal{S}_{w}$. Using a reduced decomposition of such minimal
element $w$ we have new functions $\Phi _{w}^{\Theta }$ defined in the same
way, but replacing the origin $b_{0}$ of $\mathbb{F}$ by the origin $%
b_{\Theta }$ of $\mathbb{F}_{\Theta }$, that is, 
\begin{equation*}
\Phi _{w}^{\Theta }(t_{1},\ldots ,t_{n})=\psi _{1}(t_{1})\cdots \psi
_{n}(t_{n})\cdot b_{\Theta }.
\end{equation*}%
By equivariance, $\Phi _{w}^{\Theta }=\pi _{\Theta }\circ \Phi _{w}$. This
function satisfies the required properties to be a characteristic map for
the Schubert cells $\mathcal{S}_{w}^{\Theta }$.

\begin{prop}
Take $w\in \mathcal{W}_{\Theta }^{\min }$ so that $\dim \mathcal{S}%
_{w}^{\Theta }=\dim \mathcal{S}_{w}$ and let $w=r_{1}\cdots r_{n}$ be a
reduced expression as a product of simple reflections. Let $\Phi
_{w}^{\Theta }:B_{w}\rightarrow \mathbb{F}_{\Theta }$ be defined by $\Phi
_{w}^{\Theta }=\pi _{\Theta }\circ \Phi _{w}$ and take $\mathbf{t}%
=(t_{1},\ldots ,t_{n})\in B_{w}$. Then $\Phi _{w}^{\Theta }$ is a
characteristic map for $\mathcal{S}_{w}^{\Theta }$, that is, satisfies the
following properties:

\begin{enumerate}
\item $\Phi _{w}^{\Theta }(B_{w})=\mathcal{S}_{w}^{\Theta }$.

\item $\Phi _{w}^{\Theta }(\mathbf{t})\in \mathcal{S}_{w}^{\Theta }\setminus
N\cdot wb_{\Theta }$ if and only if $\mathbf{t}\in \partial B_{w}=S^{d-1}$.

\item $\Phi |_{w}^{\Theta }{B_{w}^{\circ }}:B_{w}^{\circ }\rightarrow N\cdot
wb_{\Theta }$ is a homeomorphism, where $B_{w}^{\circ }$ is the interior of $%
B_{w}$.
\end{enumerate}
\end{prop}

\begin{proof}
This is the Proposition \ref{characteristic_map} in this generalized flag
context. The first item follows by equivariance of $\pi _{\Theta }$. The
second assertion is true because $\pi _{\Theta }(\mathcal{S}_{w}\setminus
N\cdot wb_{0})=\mathcal{S}_{w}^{\Theta }\setminus N\cdot wb_{\Theta }$ and $%
\Phi _{w}^{\Theta }=\pi _{\Theta }\circ \Phi _{w}$. The last item is a
consequence of the equality of the dimensionsx of Bruhat cells $N\cdot
wb_{0} $ and $N\cdot wb_{\Theta }$.
\end{proof}

Now we can find out the boundary maps $\partial ^{\Theta }$ with
coefficients $c^{\Theta }([w],[w^{\prime }])$, where $[w]$ denotes the the
coset $w\mathcal{W}_{\Theta }$. We have $c^{\Theta }([w],[w^{\prime }])=0$,
unless

\begin{enumerate}
\item $\dim \mathcal{S}_w^\Theta = \dim \mathcal{S}_{w^{\prime }}^\Theta + 1$
and

\item $\mathcal{S}_{w^{\prime }}^{\Theta }\subset \mathcal{S}_{w}^{\Theta }$.
\end{enumerate}

Here the inclusions among the Schubert cells are also given by the
Bruhat-Chevalley order (cf. Proposition \ref{BC-order}), namely $\mathcal{S}%
_{w^{\prime }}^{\Theta }\subset \mathcal{S}_{w}^{\Theta }$ if and only if
there is $u\in w^{\prime }\mathcal{W}_{\Theta }$ such that $u<w$. (This
follows immediately from the projections $\pi _{\Theta }\mathcal{S}_{w}=%
\mathcal{S}_{w}^{\Theta }$.) Actually, we have the following complement to
Lemma \ref{minimal_element}.

\begin{lema}
\label{lemumenos} Let $w\in \mathcal{W}_{\Theta }^{\min }$ minimal in its
coset and suppose that there exists $u\in w^{\prime }\mathcal{W}_{\Theta }$
with $u<w$ and $\dim \mathcal{S}_{w}^{\Theta }=\dim \mathcal{S}_{w^{\prime
}}^{\Theta }+1$. Then $u$ is minimal in $w^{\prime }\mathcal{W}_{\Theta }$.
\end{lema}

\begin{proof}
We have $\dim \mathcal{S}_{w}=\dim \mathcal{S}_{w}^{\Theta }=\dim \mathcal{S}%
_{u}^{\Theta }+1\leq \dim \mathcal{S}_{u}+1$. But if $u<w$ then $\dim 
\mathcal{S}_{u}\leq \dim \mathcal{S}_{w}-1$, so that $\dim \mathcal{S}%
_{w}\leq \dim \mathcal{S}_{u}+1\leq \dim \mathcal{S}_{w}$, implying that 
\begin{equation*}
\dim \mathcal{S}_{u}=\dim \mathcal{S}_{w}^{\Theta }-1=\dim \mathcal{S}%
_{u}^{\Theta }.
\end{equation*}%
Hence $u$ is minimal in its coset.
\end{proof}

\vspace{12pt}%

\noindent%

\textbf{Remark:} The assumption $\dim \mathcal{S}_{w}^{\Theta }=\dim 
\mathcal{S}_{w^{\prime }}^{\Theta }+1$ in Lemma \ref{lemumenos} is
essential. Without it there may be $u\in w^{\prime }\mathcal{W}_{\Theta }$
which is not minimal although $u<w$ and $w$ is minimal. Geometrically this
happens when $\dim \mathcal{S}_{w}=\dim \mathcal{S}_{u}^{\Theta }+1$ but $%
\dim \mathcal{S}_{w}^{\Theta }>\dim \mathcal{S}_{w^{\prime }}^{\Theta }+1$,
which may give $c\left( w,u\right) \neq 0$ and $c^{\Theta }\left(
[w],[u]\right) =0$. The following example illustrates this situation.

\vspace{12pt}

\noindent
\textbf{Example:} In the Weyl group $S_{4}$ of $A_{3}$ take $w=(12)(23)(34)$
and $\Theta =\{\alpha _{23}\}$. Then $w$ is minimal in the coset $w\mathcal{W%
}_{\{\alpha _{23}\}}$. The roots $\alpha _{12}$, $(12)\alpha _{23}=\alpha
_{12}+\alpha _{23}$ and $(12)(23)\alpha _{34}=\alpha _{12}+\alpha
_{23}+\alpha _{34}$ are positive roots mapped to negative roots by $w^{-1}$
and none of these roots lie in $\langle \Theta \rangle $. However, $%
w^{\prime }=(12)(23)=(123)$ is not minimal in its coset since $(12)<(12)(23)$
and and both belong to the same coset. Now $\dim \mathcal{S}_{(12)(23)}=\dim 
\mathcal{S}_{w}-1$ but $\mathcal{S}_{(12)(23)}^{\Theta }=\dim \mathcal{S}%
_{(12)}^{\Theta }=\dim \mathcal{S}_{w}^{\Theta }-2$.

Now if $w$ and $w^{\prime }$ belong to $\mathcal{W}_{\Theta }^{\min }$ and $%
\dim \mathcal{S}_{w}^{\Theta }=\dim \mathcal{S}_{w^{\prime }}^{\Theta }+1$
then there are the homeomorphisms $\pi _{\Theta }:N\cdot wb_{0}\rightarrow
N\cdot wb_{\Theta }$ and $\pi _{\Theta }:N\cdot w^{\prime }b_{0}\rightarrow
N\cdot w^{\prime }b_{\Theta }$. This implies that the attaching map between $%
\mathcal{S}_{w}^{\Theta }$ and $\mathcal{S}_{w^{\prime }}^{\Theta }$ defined
by $\Phi _{w}^{\Theta }=\pi _{\Theta }\circ \Phi _{w}$ and $\Phi _{w^{\prime
}}^{\Theta }\circ \Phi _{w^{\prime }}$ is the same as the attaching map
between $\mathcal{S}_{w}$ and $\mathcal{S}_{w^{\prime }}$. Hence the
coefficients for $\partial ^{\Theta }$ and $\partial $ are the same, that
is, 
\begin{equation*}
c^{\Theta }([w],[w^{\prime }])=c(w,w^{\prime }).
\end{equation*}

Hence the computation of $c^{\Theta }([w],[w^{\prime }])$ reduces to a
computation on $\mathbb{F}$.

\begin{teo}
The cellular homology of $\mathbb{F}_{\Theta }$ is isomorphic to the
homology of $\partial _{\min }^{\Theta }$ which is the boundary map of the
free module $\mathcal{A}_{\Theta }^{\min }$ generated by $\mathcal{S}_{w}$, $%
w\in \mathcal{W}_{\Theta }^{\min }$, obtained by restricting $\partial $ and
projecting onto $\mathcal{A}_{\Theta }^{\min }$.
\end{teo}

\begin{coro}
$c^{\Theta }([w],[w^{\prime }])=0$ or $\pm 2$. In particular taking
coefficients in $\mathbb{Z}_{2}$, $\partial ^{\Theta }=0$ and the $\mathbb{Z}%
_{2}$-homology of $\mathbb{F}_{\Theta }$ is freely generated by $\mathcal{S}%
_{[w]}^{\Theta }$, $[w]\in \mathcal{W}/\mathcal{W}_{\Theta }$.
\end{coro}

\vspace{12pt}%
\noindent
\textbf{Remark:} Let $w$ be minimal in its coset $w\mathcal{W}_{\Theta }$
and suppose that $u<w$ is of the form $w=ur_{\beta }$, with $\beta $ a
simple root and $\Theta =\{\beta \}$. Hence $u$ is minimal in its coset. In
fact, this conditions imply that $w\beta <0$. In fact, $w\beta =ur_{\beta
}(\beta )=-u\beta $ and $u\beta \in \Pi _{w}$. So, if $u$ is not minimal in
its coset, 
there is $\gamma >0$ such that $u^{-1}\gamma <0$ and $\gamma \in \langle
\Theta \rangle $. As $w$ is minimal, by the same fact we know that $%
w^{-1}\gamma >0$ (otherwise we would have $\Pi ^{+}\cap w\Pi ^{-}\cap
\langle \Theta \rangle \neq \emptyset $). This implies that $r_{\beta
}(u^{-1}\gamma )>0$ and hence $u^{-1}\gamma =-\beta $. Hence $\gamma
=-u\beta =ur_{\beta }(\beta )=w\beta <0 $. This is a contradiction because $%
\gamma >0$.

\vspace{12pt}

\noindent
\textbf{Example:} Let us consider the example of $G_{2}$ with $\Theta
=\{\alpha _{1}\}$. We have the following cosets 
\begin{equation*}
\mathcal{W}=\{1,r_{1}\},\{r_{2},s_{2}\},\{s_{1},r_{1}s_{2}\},%
\{r_{2}s_{1},s_{2}^{2}\},\{s_{1}^{2},r_{1}s_{2}^{2}\},%
\{r_{2}s_{1}^{2},s_{1}^{3}\}.
\end{equation*}%
The boundary maps for the minimal element in each coset is computed using the table \ref{table2}.
\begin{itemize}
\item $\partial_5 \mathcal{S}_{r_{2}s_{1}^{2}}=-2%
\mathcal{S}_{s_{1}^{2}}$;
\item $\partial_4 \mathcal{S}_{s_{1}^{2}}=0$;
\item $\partial_3 \mathcal{S}_{r_{2}s_{1}}=0$;
\item $\partial_2 \mathcal{S}_{s_{1}}=-2\mathcal{S}_{r_{2}}$;
\item $\partial_1 \mathcal{S}_{r_{2}}=0$.
\end{itemize}
Hence
\begin{itemize}
\item $H_{5}(\mathbb{F}_{\alpha _{1}},\mathbb{Z})=0$ (in particular $\mathbb{%
F}_{\{\alpha _{1}\}}$ is not orientable).

\item $H_4(\mathbb{F}_{\alpha_1},\mathbb{Z})= \mathbb{Z}_2$.

\item $H_3(\mathbb{F}_{\alpha_1},\mathbb{Z}) = \mathbb{Z} $.

\item $H_2(\mathbb{F}_{\alpha_1},\mathbb{Z}) = 0$.

\item $H_1(\mathbb{F}_{\alpha_1},\mathbb{Z}) = \mathbb{Z}_2$.
\end{itemize}

As another source of examples, we refer to the papers \cite{Rab16} and \cite{RL182} which computes the coefficients of the isotropic and orthogonal Grassmannians.

\bibliographystyle{amsplain}

\bibliography{biblio}

\end{document}